\documentclass{article}
\usepackage{amsmath,amssymb,amsthm}
\usepackage[svgnames, table]{xcolor} % Enabling colors by their 'svgnames'
\usepackage[ruled,vlined]{algorithm2e}
\usepackage{setspace}
\usepackage{xr-hyper}
\usepackage[colorlinks=true, linkcolor=blue, citecolor=blue, filecolor=blue, runcolor=blue, urlcolor = blue]{hyperref}

\usepackage{fullpage}
\newtheorem{definition}{Definition}[section]
\newtheorem{theorem}{Theorem}[section]
\newtheorem{proposition}{Proposition}[section]

\newtheorem{corollary}{Corollary}[section]
\newtheorem{remark}{Remark}[section]
\newtheorem{example}{Example}[section]

\usepackage{graphicx,calc}
\DeclareGraphicsExtensions{.png}

\begin{document}

\title{The Virtual Spectrum of Linkoids and Open Curves in 3-space}

\author{Kasturi Barkataki \and Louis H. Kauffman \and Eleni Panagiotou}

\date{\today}

\maketitle

\begin{abstract}
%% Text of abstract
The entanglement of open curves in 3-space appears in many physical systems and affects their material properties and function. A new framework in knot theory was introduced recently, that enables to characterize the complexity of collections of open curves in 3-space using the theory of knotoids and linkoids, which are equivalence classes of diagrams with open arcs. In this paper, new invariants of linkoids are introduced via a surjective map between linkoids and virtual knots. This leads to a new collection of strong invariants of linkoids that are independent of any given virtual closure. This gives rise to a collection of novel measures of entanglement of open curves in 3-space, which are continuous functions of the curve coordinates and tend to their corresponding classical invariants when the endpoints of the curves tend to coincide.
\vspace{2pc}

\end{abstract}

\section{Introduction}
The entanglement of open curves in 3-space is a common problem of interest in many physical systems, such as polymers, textiles and crystals \cite{Fukuda2023, Liu2018,Edwards1967}.
Even  though the study of topological complexity of simple closed curves in space is studied thoroughly in knot theory, characterizing the complexity of open curves in 3-space has been a long-standing open problem in mathematics. A reason is that, the traditional topological invariance ideas are not applicable in the context of open curves in 3-space. Recent works have introduced a framework upon which knot theory can be extended to study open curves in 3-space without any approximation schemes \cite{Panagiotou2020b, Panagiotou2021, Barkataki2022}. In this paper, new measures of entanglement of collections of open curves in 3-space
are introduced via a map between linkoids and virtual knots/links.

A knot (or link)  consists in a collection of closed curve(s) embedded in three dimensional space. Knots
and links are classified, with respect to their complexity, by topological invariants, usually of the form of integer valued functions or polynomials with integer coefficients \cite{Jones1985,Freyd1985,Jones1987,Kauffman1987,Kauffman1990,Przytycki1987}.  A topological invariant is a function on the space of knots or links that is invariant under continuous deformations of the embeddings that do not allow self intersections. 

 Knotoids/linkoids  are open ended knot diagrams which can be classified under a diagrammatic notion of equivalence similar to that of knots, but with some restrictions regarding the isotopy moves on the endpoints of the diagram \cite{Turaev2012,Gugumcu2017}. As mentioned in \cite{Turaev2012}, the study of knotoids is closely related to that of virtual knots. A virtual knot diagram is one with some of the crossings being of a different nature, called virtual and can be classified using a diagrammatic theory which is very similar to the handling of classical knot/link diagrams \cite{Kauffmann1999}. There exist many invariants of knotoids, several of them via virtual knots, such as height, odd writhe , the Jones polynomial, the parity bracket polynomial, the affine index polynomial and the arrow polynomial. Even though it is often assumed that invariants of linkoids would follow directly from those of knotoids, as is the case for knots and links, it has been substantially harder to define invariants of linkoids.
 
 In this paper, we give a general connection between knotoids/linkoids and virtual links. This leads to new invariants of linkoids via virtual closure, namely - height, genus, odd writhe, Jones polynomial, arrow polynomial and affine-index polynomial. By accounting for all different virtual closures, new invariants of linkoids are defined - the virtual spectrum and average spectral invariants, which are stronger and are independent of any given virtual closure.
 
Although open curves in 3-space are not knotted or linked in the topological sense, they can form complex conformations, which are called entangled. Entanglement of open curves can also  be measured via topological/geometrical measures, which are either real numbers or polynomials with real coefficients, that are continuous functions of the curve coordinates and can detect knotting and threading \cite{Barkataki2022, Gauss1877, Panagiotou2020b, Panagiotou2021}. A projection of a collection of curves in 3-space can be a closed or open knot/link diagram.
In \cite{Panagiotou2020b,Barkataki2022} it was shown that a rigorous measure of complexity of a collection of open curves in 3-space is given by taking the average of an invariant of a projection of the curves over all possible directions of projection. 

In this paper, the special construction of the virtual closure of open curves in 3-space is introduced, which leads to the definition of new measures of entanglement of collections of open curves in 3-space. These are continuous functions of the curve coordinates and tend to the classical topological invariant as the endpoints coincide.
 
The paper is organised as follows : Section \ref{sec-basics} gives an exposition into some basic definitions and properties associated with virtual knots. Section \ref{sec-vir-linkoid} introduces the generalized virtual closure of linkoids and makes a rigorous connection between the theory of linkoids and the theory of virtual knots. Section \ref{sec-inv-via-vc} discusses the emergence of new invariants of linkoids via the virtual closure introduced in Section \ref{sec-vir-linkoid}. Section \ref{sec-vir-spec} discusses the virtual spectrum of linkoids and new invariants of linkoids independent of any given virtual closure. Section \ref{sec-open-curves} introduces the virtual closure of open curves in 3-space and  new measures of entanglement of open curves in 3-space and discusses their properties. Section \ref{sec-conc} presents the conclusions of this study.

\section{Basics on Virtual Knots and Links}
\label{sec-basics}

This section gives an overview of the basic definitions pertaining to virtual knots and links and their properties \cite{Kauffmann1999, Gugumcu2017, Carter2000, manturov2012virtual}.

\begin{definition}(\textit{Virtual knot/link} and \textit{virtual knot/link diagram})
A \textit{virtual knot/link diagram} consists of generic closed curves in
$\mathbb{R}^2$ (or $S^2$) such that each crossing is either a classical crossing with over and under arcs, or
a virtual crossing without over or under information. Virtual knot/link diagrams are classified using the generalized Reidemeister moves, which include the classical Reidemeister moves and the virtual Reidemeister moves (See Figure \ref{vr1}) and the forbidden moves (See Figure \ref{vr2}). A \textit{virtual knot/link} is defined as an equivalence class of virtual knot/link diagrams under the generalised Reidemeister moves. In conjuction with the generalized Reidemeister moves, a segment of a diagram, consisting of a sequence of consecutive virtual crossings, can be excised and a
new connection be made between the resulting free ends. If the new connecting segment intersects
the remaining diagram (transversally) then each new intersection is taken to be virtual. Such an
excision and reconnection is called a \textit{detour move}  (See Figure \ref{vr2}).
 \label{virtual_knot/link}
\end{definition}

\begin{figure}[ht!]
    \centering
\includegraphics[scale=0.6]{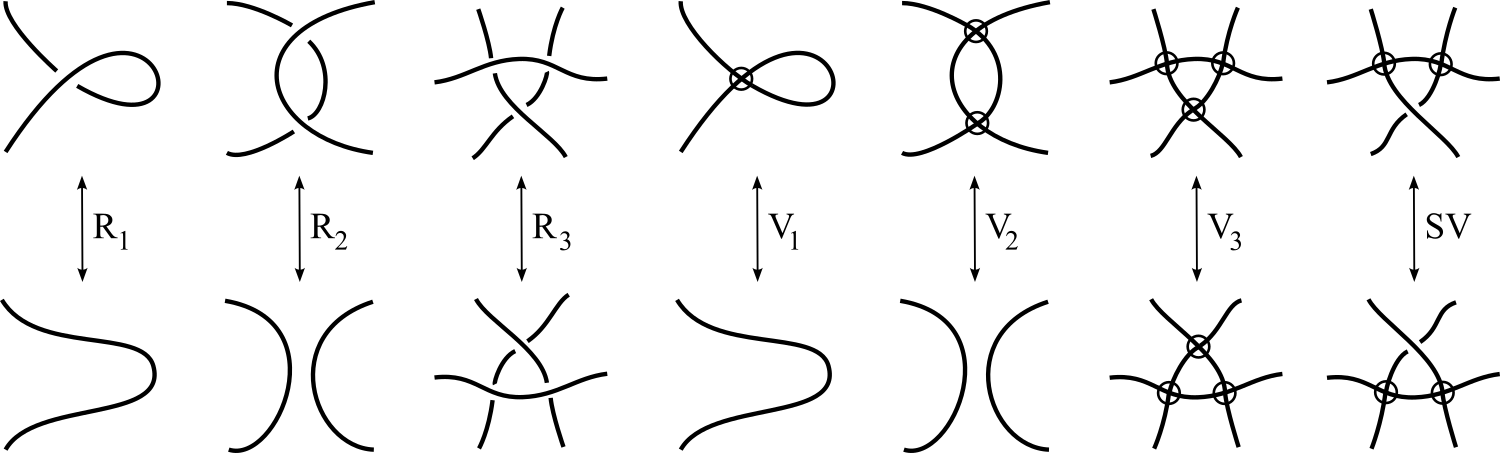}
    \caption{Generalised Reidemeister moves on virtual knots/links : the classical Reidemeister moves, $R_1$, $R_2$,
$R_3$; the virtual moves, $V_1$, $V_2$, $V_3$; and the semi-virtual move $SV$.}
    \label{vr1}
\end{figure}

\begin{figure}[ht!]
    \centering
\includegraphics[scale=0.6]{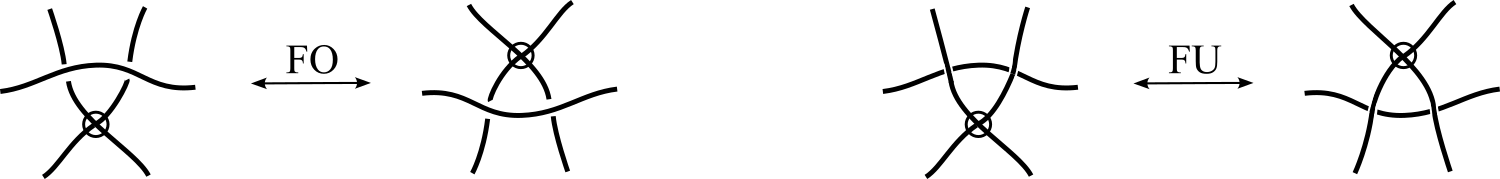} \hspace{1.5cm} \raisebox{-5pt}{\includegraphics[scale=0.18]{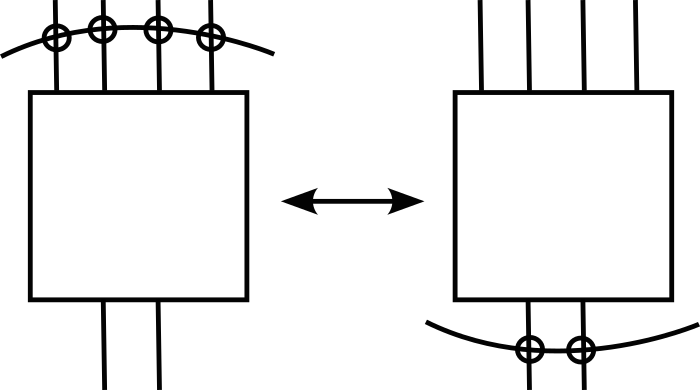}}
    \caption{Forbidden moves on virtual knots/links : the move involving an over-strand, FO and the move involving an under-strand, FU; and the detour move of an arc in a virtual knot diagram.}
    \label{vr2}
\end{figure}

Given a virtual knot/link diagram, $VL_g$, with $g$ virtual crossings, there is a canonical way to associate it with a knot/link in a thickened surface : Each virtual crossing in the diagram can be considered as a shorthand for a detour of one of the arcs involved in the crossing through a 1-handle that has been attached to the 2-sphere of the original diagram. The two choices (above or below) for the 1-handle detour are homeomorphic to each other (as abstract manifolds with boundary). This gives an embedding (the canonical embedding) of a collection of circles into the sphere with $g$ handles. Any other embedding, which is stably equivalent to the canonical embedding, is considered to be a valid embedding of $VL_g$. In fact, two virtual knot/link diagrams are equivalent if and only if their corresponding surface embeddings are stably equivalent \cite{Kauffmann1999,Carter2000}. 

%\begin{definition}\textit{(Stable equivalence)}
%Two surface embeddings of knots/links  are said to be \textit{stably equivalent} if one can be obtained from another by isotopy in the thickened surfaces, homeomorphisms of the surfaces and the addition or subtraction of
%empty handles. The operation of removing empty handles is called destabilization and this operation does not change the knot but reduces the genus of the surface on which it is embedded.
%\end{definition}
\begin{figure}[ht!]
    \centering
     \includegraphics[scale=0.4]{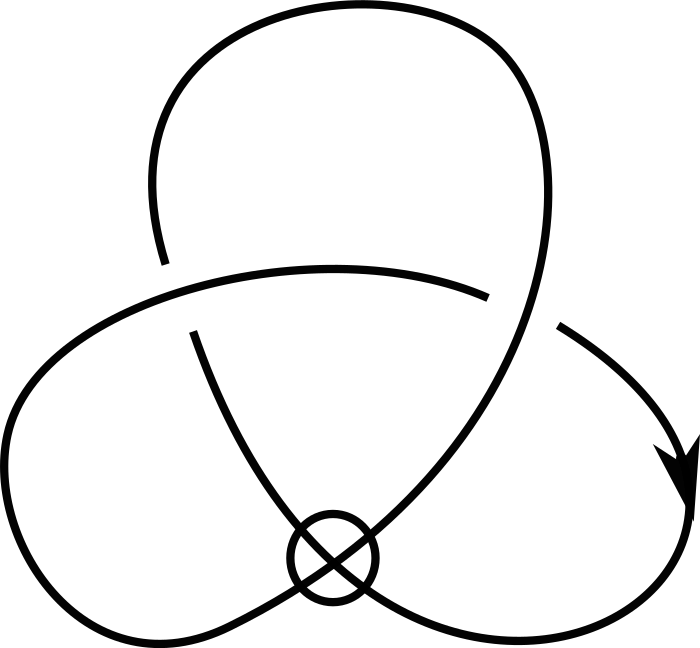} \hspace{1cm} \includegraphics[scale=0.4]{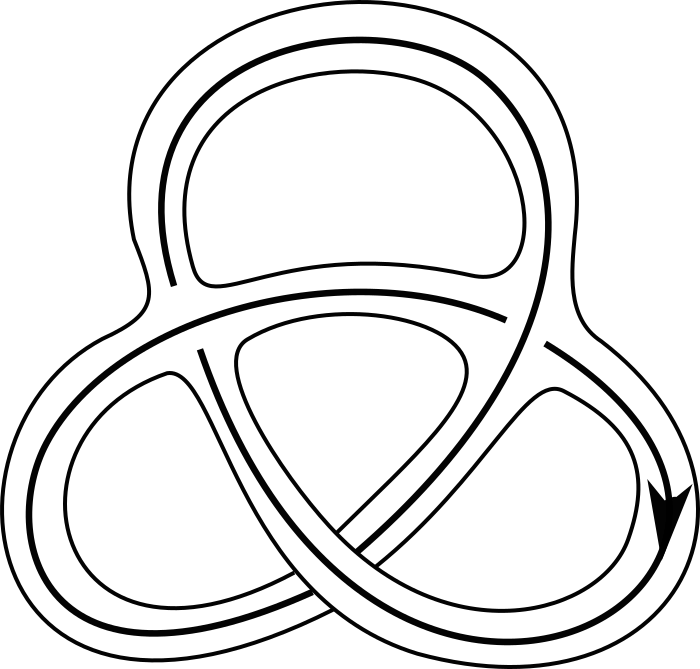} \hspace{1cm}
      \includegraphics[scale=0.4]{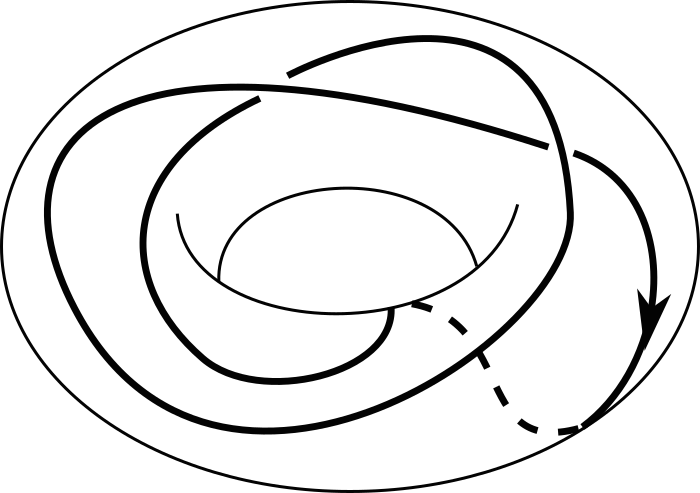}
    \caption{(Left) A diagram of a virtual trefoil knot, (Center) Its ribbon neighborhood representation, (Right) Its surface representation.}
    \label{ribbon1}
\end{figure}
\begin{remark}
A stably equivalent way to realize virtual knots is to form a ribbon neighborhood surface/abstract link diagram  for the given virtual
knot or link diagram \cite{Kauffman2006virtual}.  In such a diagram, the classical crossings are represented as diagrammatic crossings in disks, which are connected by ribbons, while the virtual crossings are
represented by crossings between two ribbons that pass over/under one another without interacting. The ribbon neighborhood diagram is independent of the choice of over or under information for the ribbon strips. By gluing discs to the boundary circles of a ribbon neighborhood diagram of a virtual knot, we obtain a unique surface representation of the virtual knot (See Figure \ref{ribbon1}). 
\label{ribbon-remark}
\end{remark}  
The virtual crossing number of a virtual knot is defined as follows \cite{manturov2012virtual,satoh2012crossing}:

\begin{definition}(\textit{Virtual crossing number}, and \textit{Minimal virtual knot/link diagram})
The  \textit{virtual crossing number}, $\mathsf{cr}_v(VL)$ of a virtual knot/link $VL$ is the minimum of the number of virtual crossings over all diagrams of $VL$. A diagram corresponding to a virtual knot/link $VL$ is said to be \textit{minimal} if the number of virtual crossings in the diagram  is equal to $\mathsf{cr}_v(VL)$. The minimal diagram of a virtual knot/link is unique upto classical Reidemeister moves.
\label{vir}
\end{definition}

A virtual knot/link can also be studied as a knot in a thickened surface through its \textit{minimal representation} \cite{Kuperberg2003, dye2009virtual}, defined as follows:

\begin{definition}(\textit{Genus of a virtual knot/link} and \textit{Minimal representation of a virtual knot/link})
 The \textit{genus}, $g$, of a virtual knot/link is defined as the genus of its minimal embedding surface. An embedding of a virtual knot/link is said to be its \textit{minimal representation}
if it is an embedding in an abstract surface, $S$, whose genus is equal to the genus of the given virtual knot. The minimal representation of a virtual knot/link is unique up to homeomorphism of the surface $S$. 
\end{definition}

\begin{remark}
Let $D(V)$ be a diagram of a virtual knot, $V$ and $a(D(V))$ be its ribbon neighborhood diagram. The surface representation of $V$, obtained by gluing discs to the two boundary circles in $a(D(V))$, realizes the unique (up to homeomorphism) minimal embedding surface of $V$.
\end{remark}

\section{Linkoids and Virtual Links}
\label{sec-vir-linkoid}

In this section an analogy is drawn between virtual knots/links and \textit{virtual closures} of knotoids/linkoids. Linkoids have been typically studied as diagrammatic objects and they can be thought of as projections of open curves in 3-space \cite{Barkataki2022,Gugumcu2017,Turaev2012,Gugumcu2017b,Gugumcu2021parity, Manouras2021}. Throughout this section, the symbol $\Sigma$ is used to denote a surface on which a linkoid diagram lies. The results can be generalized for any $\Sigma$ but in this manuscript, it is assumed that  $\Sigma=S^2 = \mathbb{R}^2 \cup \infty$.

\begin{definition}\label{linkoid}(\textit{Knotoid/linkoid diagram} and \textit{Knotoid/linkoid})
 A \textit{knotoid/linkoid diagram} $L$ with $n \in \mathbb{N}$ components in $\Sigma$ is a generic immersion of $\bigsqcup_{i=1}^{n}[0,1]$ in the interior of $\Sigma$ whose only singularities are transversal double points endowed with over/undercrossing data. These double points are called the crossings of $L$. The immersion of each $[0,1]$ is referred to as a component of the knotoid/linkoid diagram and the images of $0$ and $1$ under this immersion are called the foot and the head of the component, respectively. These two points are distinct from each other and from the double points; they are called the endpoints of the component. The diagram $L$ has a total of $2n$ endpoints. A \textit{knotoid/linkoid} is an equivalence class of knotoid/linkoid diagrams up to the equivalence relation induced by the three Reidemeister moves and isotopy. It is forbidden to pull the strand adjacent to an endpoint over/under a transversal strand  (See Figure \ref{linkoid-moves}).
\end{definition}
\begin{figure}[ht!]
    \centering
    \includegraphics[scale=0.6]{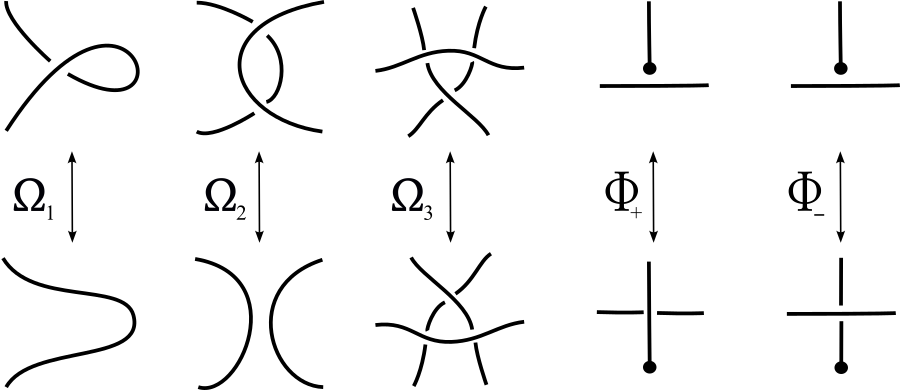}
    \caption{Omega moves (Reidemeister moves) and forbidden moves ($\Phi_+, \Phi_-$) on linkoid diagrams.}
    \label{linkoid-moves}
\end{figure}

Given a linkoid $L$ with $n$ components, a choice of labelling can be made for the $2n$ endpoints with numbers from the set $\{1, 2, \cdots, 2n\}$, without repetition. The initial connectivity among endpoints of $L$ allows us to define a \textit{strand permutation} as follows :

 \begin{definition}\label{def-link-perm}(\textit{Strand Permutation})
Let the endpoints of a linkoid diagram, $L$,  with $n$ components, be denoted by labels from the set $E=\{1, 2, \cdots, 2n\}$, without repetition. The \textit{strand permutation} of $L$ is defined to be the element, $\tau \in S_{2n}$, such that, for any $i\in E$, i.e. the head/foot of a component, $\tau(i)\in E$ is the corresponding foot/head and $\tau(i)\neq i$. It follows that $\tau$ is of order $2$. In other words, $i$ and $\tau(i)$ are labels for the two endpoints of an open component (strand) in the linkoid.
\end{definition}

\begin{figure}[ht!]
    \centering
    \includegraphics[scale=0.5]{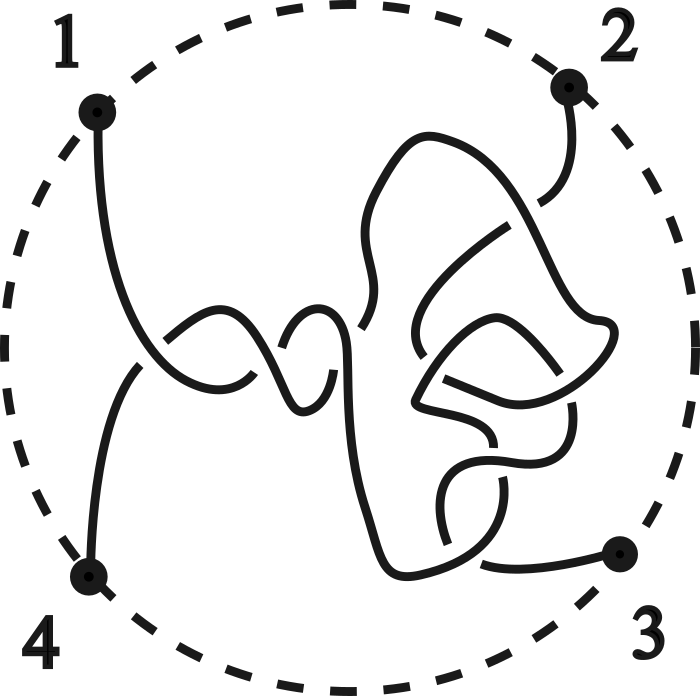}\hspace{50pt}\includegraphics[scale=0.5]{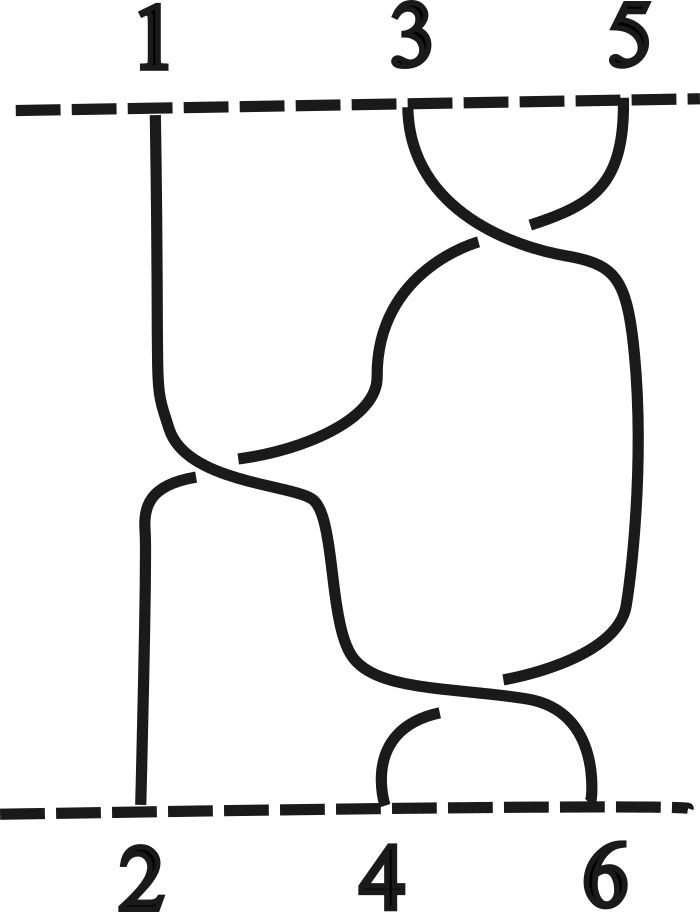}\hspace{40pt} \raisebox{0.1\totalheight}{\includegraphics[scale=0.25]{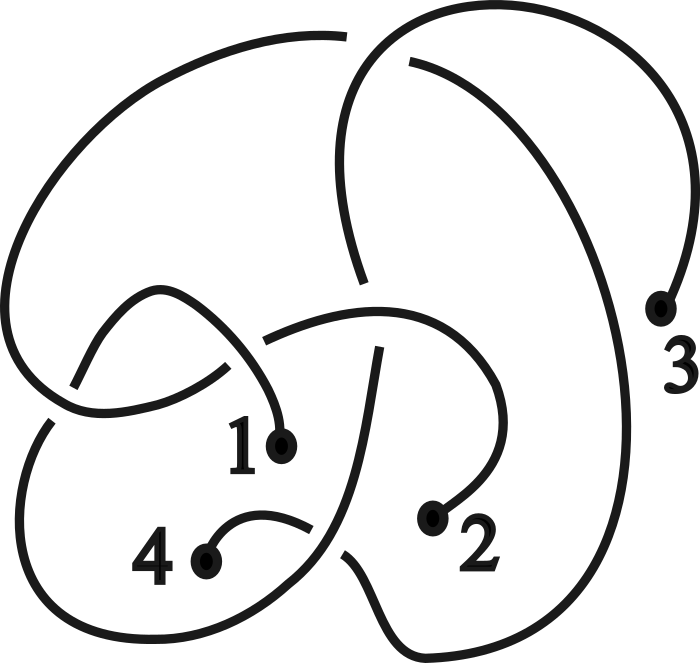}}\\\vspace{20pt}
    \includegraphics[scale=0.5]{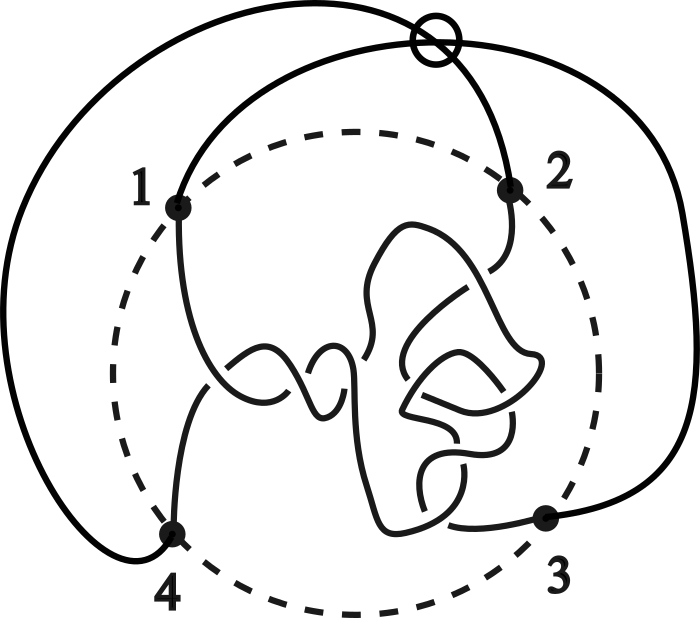}\hspace{25pt}\includegraphics[scale=0.5]{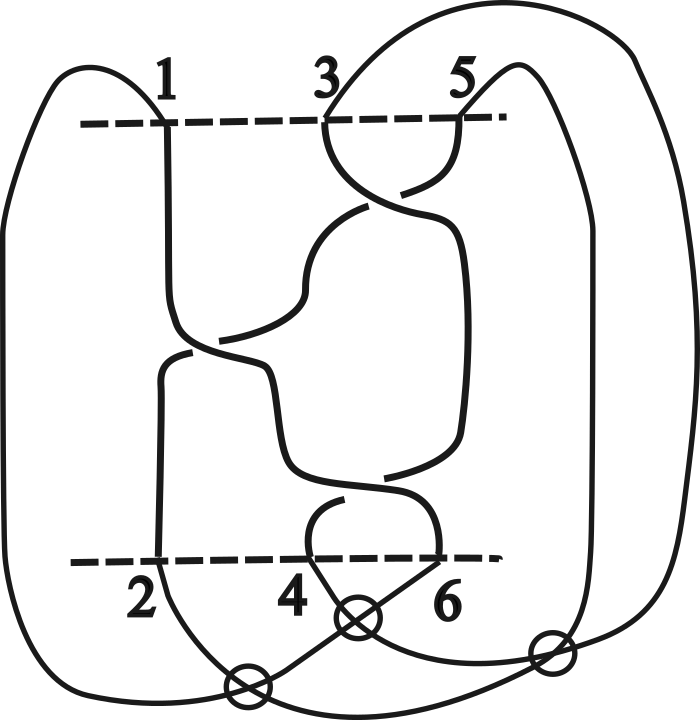}\hspace{20pt} \raisebox{0.1\totalheight}{\includegraphics[scale=0.25]{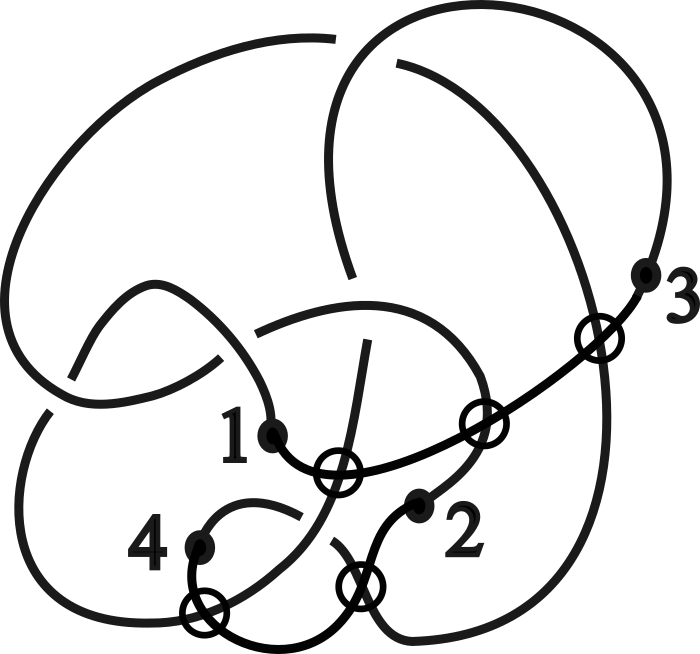}}
    \caption{(Above) The labelled endpoints of any linkoid determine a strand permutation. Examples of a tangle, a braid and a general linkoid are shown above with strand permutations $(1 \quad 3)(2 \quad 4)$, $(1 \quad 6)(3 \quad 4)(2 \quad 5)$ and $(1 \quad 3)(2 \quad 4)$, respectively. (Below) The virtual closure of the tangle, the braid and the linkoid according to their respective strand permutations.}
    \label{fig-linkoid-perm-1}
\end{figure}

By definition, the strand permutation for a linkoid with $n$ components, is a product of $n$ disjoint transpositions (See Figure \ref{fig-linkoid-perm-1}). 
In the following, given a permutation $\tau$ in $E=\{1, 2, \cdots, 2n\}$, the notation $L_\tau$ is used to represent a labelled linkoid with $n$ components with the head-foot labelling implied by its strand permutation, $\tau$.

\subsection{Generalized Virtual Closure of Linkoids}
In this section, the relationship between multi-component linkoids and virtual links is examined for the first time and a general framework for studying the map between linkoids and virtual knots/links is introduced. The relationship between knotoids (single component) and virtual knots has been explored before \cite{Turaev2012, Gugumcu2017}. It was shown that, by introducing a closure arc to a knotoid diagram, which only creates virtual crossings, one can obtain a virtual knot diagram. The extension of the approach presented in \cite{Turaev2012, Gugumcu2017} to the case of linkoids with more than 1 component, is substantially more complex and it is studied for the first time in this manuscript.

Indeed, notice that, there are multiple ways to close a linkoid of more than 2 components and it is possible to obtain non-equivalent virtual knots/links, for the same linkoid, based on the choice of closure of the endpoints. To make this precise, the following definition of a \textit{closure permutation} will be useful :

\begin{definition}(\textit{Closure Permutation}) Let the endpoints of a linkoid diagram, $L$,  with $n$ components, be denoted by labels from the set $E=\{1, 2, \cdots, 2n\}$, without repetition. A \textit{closure permutation} of $L_{\tau}$, where $\tau$ denotes the strand permutation of $L$, is any element,  $\sigma \in S_{2n}$,  such that, $\sigma^2 = \operatorname{id}$ and $\sigma(i) \neq i$, for all $i \in E$. 
\label{def_EP}
\end{definition}

\begin{remark}
 Definitions \ref{def-link-perm} and \ref{def_EP} imply the following : (i)  The strand permutation is a special case of a closure permutation. (ii) Any closure permutation can be expressed as the product of $n$ disjoint transpositions. (iii) Strand permutation (resp. closure permutation) of a linkoid is synonymous with the term head-foot pairing (resp. pairing combination) introduced in \cite{Barkataki2022}.
\end{remark}

\begin{figure}[ht!]
    \centering
    \includegraphics[scale=0.55]{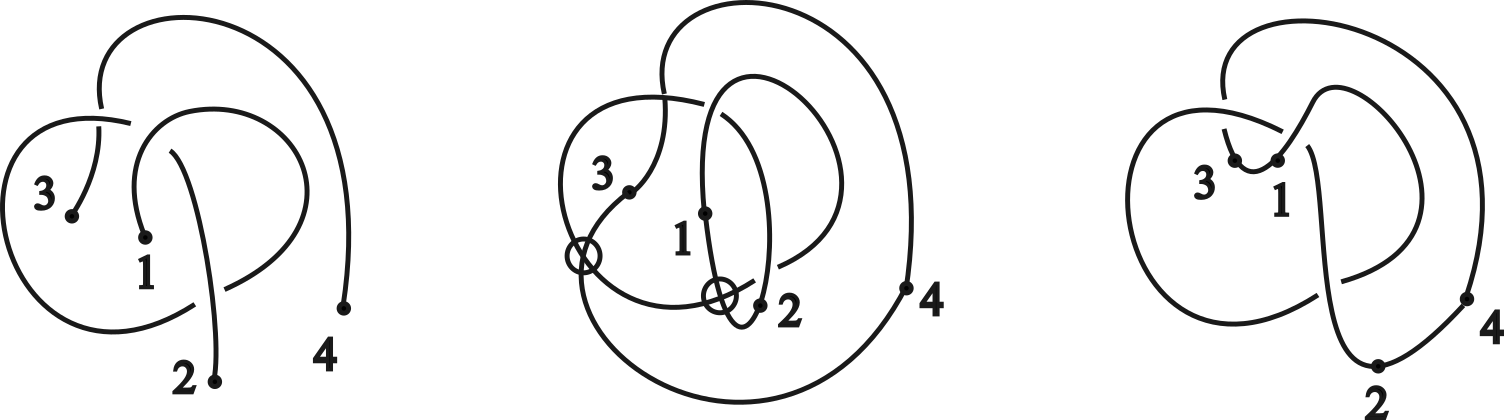}
    \caption{The virtual closure of a linkoid can give rise to non-equivalent virtual knots depending on the choice of the closure permutation, as seen in the above example. (Left) An example of  a linkoid, $L$, with $\tau = (1 \quad 2)(3 \quad 4)$. (Center) The virtual closure, $(L_\tau, \tau)$ of $L$, with respect to $\tau$, gives rise to a virtual link. It is in fact the \textit{strand closure} of $L$. (Right) Whereas, the virtual closure, $(L_\tau, \sigma)$  with respect to another closure permutation, $\sigma = (1 \quad 3)(4 \quad 2)$, gives rise to a trefoil.}
    \label{closures1}
\end{figure}

Any virtual closure of a linkoid diagram, with respect to a given closure permutation, can be defined as follows:

\begin{definition}(\textit{Virtual Closure of a linkoid diagram with respect to $\sigma$} and \textit{Strand Closure})
Let $L_{\tau}$ be a linkoid diagram  with $n$ components; $\tau$ be its strand permutation and $\sigma \in S_{2n}$ be a closure permutation of $L_\tau$. The \textit{virtual closure} of $L_\tau$ with respect to  $\sigma$, denoted $(L_\tau,\sigma)_v$, is defined to be the virtual knot/link obtained by introducing a closure arc between $i$ and $\sigma(i)$, for every $i \in \{1, 2, \cdots, 2n\}$, consisting entirely of virtual crossings. The special case of the virtual closure  with $\sigma = \tau$ will be called the \textit{strand closure} of $L$. (See Figure \ref{closures1}). 
 \label{closure}
\end{definition}

% subtle difference from classical links
Depending on the number of open-ended arcs, the strand permutation and the choice of closure permutation, the number of components of a virtual closure of a linkoid can be determined by its associated linkoid and closure permutations, by generalizing the notion of \textit{segement cycle} (introduced in \cite{Barkataki2022}), as follows:

\begin{definition}(\textit{Segment cycle of a linkoid with respect to a closure permutation})
Let $L_\tau$ be a linkoid with $n$ components, with strand permutation $\tau$ and a closure permutation $\sigma$. Let $E= \{ 1, 2, \cdots, 2n\}$ denote the set of the labelled endpoints of $L_\tau$. A \textit{segment cycle} of $L_\tau$ with respect to the closure permutation $\sigma$, is defined to be an orbit of a point in $E$, under the action of the group $\langle \tau, \sigma \rangle$ on $E$.
\label{def-seg-linkoid}
\end{definition}

\begin{figure}[ht!]
    \centering
    \includegraphics[scale=0.35]{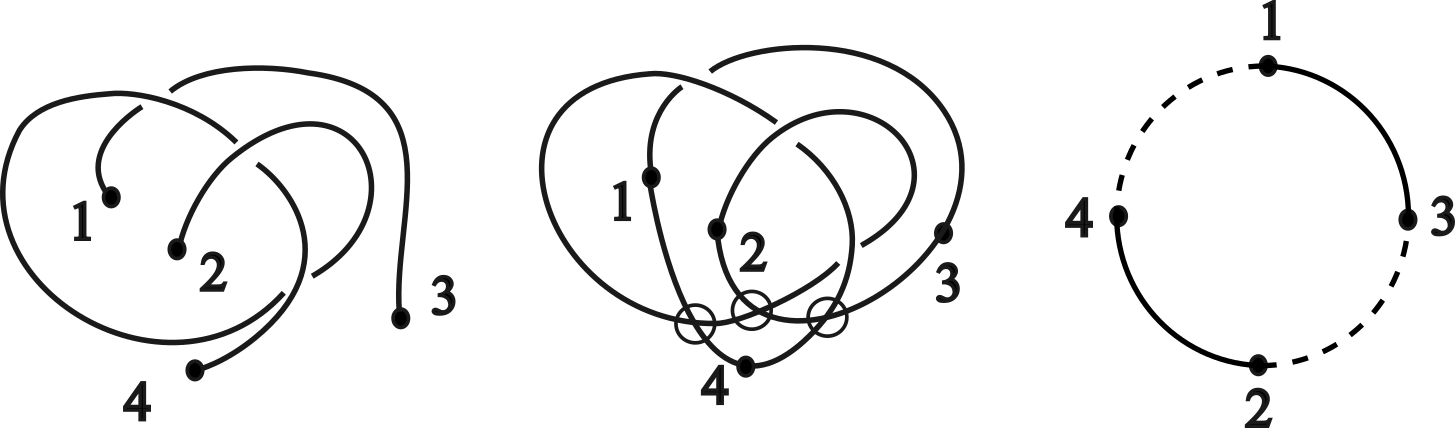}
    \caption{(Left) A linkoid, $L$, with 2 components and strand permutation, $\tau = (1 \quad 3)(2\quad 4)$; (Center) The virtual closure, $(L_\tau,\sigma)_v$, with respect to $\sigma = (1 \quad 4)(2 \quad 3)$; (Right) the unique segment cycle in $L$, with respect to $\sigma$.}
    \label{linkoid-compo}
\end{figure}

The definition of a segment cycle  and the defining properties of a group guarantee that the set of segment cycles of (points in) $E$ under the action of 
$\langle \tau, \sigma \rangle$ form a partition of $E$. The set of all segment cycles of the linkoid under the action of $\langle \tau, \sigma \rangle$ is written as the  quotient of the action, i.e, $\displaystyle E/\langle \tau, \sigma \rangle$. The number of components in a closure of a linkoid can be detremined by the segment cycles as follows :

\begin{proposition}\label{prop-seg-cycles}
Let $L_\tau$ be a linkoid with $n$ components, with strand permutation $\tau$ and a closure permutation $\sigma$. Let $E= \{ 1, 2, \cdots, 2n\}$ denote the set of the labelled endpoints of $L_\tau$. Let $G=\langle \tau, \sigma \mid \tau^2=\sigma^2=1\rangle$. The number of distinct segment cycles, $\displaystyle |E/G|$ of $L_\tau$, with respect to $\sigma$ is given as 
$$\displaystyle |E/G| = \frac{1}{2m}\sum_{g \in G} |E^g|$$
\noindent where $m=|\tau \sigma|$ is the order of $\tau \sigma \in G$ and $E^g$ denotes the set of elements in $E$ that are fixed by $g \in G$ i.e. $\displaystyle E^g = \{ x \in E \quad | \quad g.x = x \}$.
\end{proposition}
\begin{proof}
$G$ can alternatively be expressed as, $\langle \tau \sigma, \sigma \mid (\tau\sigma)^m=\sigma^2= (\sigma \tau \sigma)^2= 1\rangle$. Thus, $G$ is isomorphic to the dihedral group $D_{m}$. By Definition \ref{def-seg-linkoid}, segment cycles of $E$ form a partition of $E$ under the action of $G$. Therefore, by Burnside's counting theorem, the number of distinct segment cycles of $L_\tau$ with respect to the closure permutation, $\sigma$ is given as,
\begin{equation}
   \displaystyle |E/G| = \frac{1}{|G|}\sum_{g \in G} |E^g|  = \frac{1}{2m}\sum_{g \in \langle \tau, \sigma \rangle} |E^g|.
\end{equation}
\end{proof}

\begin{corollary}\label{corr-seg-virconn}
Let $L_\tau$ be a linkoid with strand permutation $\tau$ and a closure permutation $\sigma$. The number of components in the virtual closure $(L_\tau,\sigma)_v$ is equal to the number of segment cycles of $L_\tau$ with respect to $\sigma$.
\end{corollary}

\begin{proof}
Let $E$ be the set of all labelled endpoints of the linkoid, $L_\tau$. If $i\neq j\in E$ lie in the same component, then $i,j$ are either in the same linkoid component or they belong to different components that are connected via closure, thus they are in the same orbit under the action of the group $\langle \tau, \sigma\rangle$ on $E$, which is the same segment cycle. Similarly, a segment cycle defines a component of the virtual closure, since, if $i,j$ are in different components (different linkoid components and not connected via closure), they are not related by the action of $\langle \tau, \sigma\rangle$ on $E$
(See Figure \ref{linkoid-compo}).
\end{proof}

\begin{remark}
Braids and tangles are special types of linkoids.  Hence, Definitions  \ref{linkoid}-\ref{closure}, Proposition \ref{prop-seg-cycles} and Corollary \ref{corr-seg-virconn} apply naturally to braids and tangles.  Given any braid, $B$, with $n$ components, is it interesting to consider the two very special kinds of braid closures - the standard braid closure and the strand closure. 
As a convention, the endpoints of a braid are numbered, in order, from the top bar to the bottom bar. The standard braid closure is obtained by introducing a closure arc between $i$ and $i+n$, for all $i \in \{1, 2, \ldots, n\}$. On the other hand, the strand closure is obtained by introducing a closure arc between $i$ and $\tau(i)$, for all $i \in \{1, 2, \ldots, 2n\}$, where $\tau$ is the strand permutation associated with the braid (See Figure \ref{two-kinds-BC}). 
\end{remark}
\begin{figure}[ht!]
    \centering
    \includegraphics[scale=0.5]{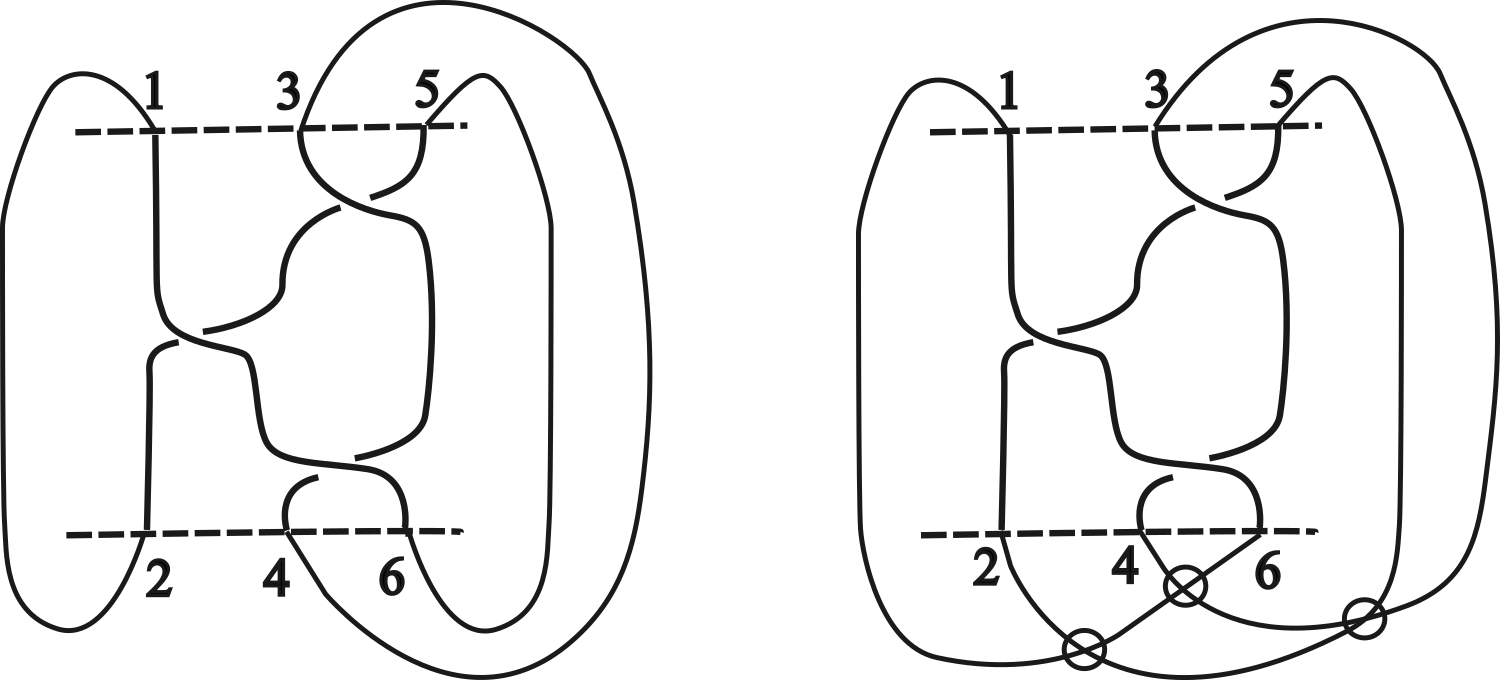}
    \caption{(Left) Standard braid closure and (Right) Strand closure of a braid on 3 strands.}
    \label{two-kinds-BC}
\end{figure}

\begin{proposition}\label{rem-clos1}
 Let $L$ be a linkoid diagram with strand permutation, $\tau$, and a closure permutation, $\sigma$. The virtual closure, $(L_\tau,\sigma)_v$, is independent of the immersion class of the closure arcs.  
 \end{proposition}
 \begin{proof}
    Let $A$ and $B$ be two diagrams, of the virtual closure of the given linkoid, with different immersion classes of the closure arcs. The diagram $B$ can be obtained by detour moves on the closure arcs in $A$.
 \end{proof}

\begin{remark}
    For a linkoid, $L_\tau$, the virtual closure with respect to the strand permutation, $(L_\tau, \tau)$, corresponds to the component-wise closure of $L_\tau$. Thus, the virtual closure of a linkoid corresponding to the strand permutation (i.e. $\sigma = \tau$) has the same number of components as the linkoid.
\end{remark}

\begin{proposition}\label{prop-L-sig-iso}
    Let $L_{\tau}^{1}, L_{\tau}^{2}$ denote two equivalent linkoid diagrams with strand permutation $\tau$ and let $\sigma$ denote a closure permutation. Then their virtual closures with respect to $\sigma$, $(L_{\tau}^{1},\sigma)$ and $(L_{\tau}^{2},\sigma)$ are equivalent.
\begin{proof}
It follows by performing a sequence of Reidemeister moves on the virtual knot/link, $(L_\tau^1, \sigma)$, which transforms $L_\tau^1$ to $L_\tau^2$.
\end{proof}
\end{proposition}

 \begin{remark}
 \label{linkoid-thickened-surf}
 Let $L$ be a linkoid diagram with strand permutation, $\tau$, and a closure permutation, $\sigma$ and let 
 $(L_\tau,\sigma)_v$ be the corresponding
 unique virtual knot/link obtained via virtual closure.  $(L_\tau,\sigma)_v$ can be interpreted  in terms of links in thickened surfaces or on ribbon-neighborhoods  (See Section \ref{sec-basics}).
\end{remark}

\begin{definition}(\textit{Minimal virtual closure  } and \textit{Minimal virtual representative of a linkoid diagram with respect $\sigma$}) A virtual closure of a linkoid diagram, $L_\tau$, with respect to a closure permutation, $\sigma$,  is said to be \textit{minimal} if no virtual crossing can be further removed via a sequence of generalized Reidemeister moves and the 0-move on $S^2$. The resultant virtual knot/link,  $(L_\tau,\sigma)_v$, is called the \textit{minimal virtual representative} of $L_\tau$ with respect to $\sigma$.
\label{def-min-vir-clo}
\end{definition}

\begin{definition}(\textit{Link-type linkoid with respect to $\sigma$})
A linkoid $L_\tau$ is said to be link-type with respect to a closure permutation, $\sigma \in S_{2n}$, if the minimal representative diagram of $(L_\tau,\sigma)_v$ has no virtual crossings.  Thus, the virtual closure of $L_\tau$, with respect to $\sigma$, gives rise to a classical link, $\mathcal{C}$. Then, the link-type linkoid, $L_\tau$, is said to be of link-type $\mathcal{C}$ with respect to $\sigma \in S_{2n}$.
\label{def-link-type}
\end{definition}

\subsection{The Virtual Closure Map}
The virtual closures of linkoids define a map between linkoids and virtual links which can be formalized by the means of the \textit{virtual closure map}, as follows:

\begin{definition}(\textit{Virtual closure map})
Let for every $n \in \mathbb{N}$, $\mathcal{A}_n$ be the set of all linkoids with $n$ components and $H_n = \{ \sigma \in S_{2n} : \sigma^2(i) = i, \sigma(i) \neq i , \quad \forall i \in E \}$, where $E=\{ 1, 2, \ldots , 2n\}$. Let $\mathcal{V}$ be the set of all virtual knots/links. The \textit{virtual closure map}, $\varphi$, is defined as,  
\begin{equation*}
\begin{split}
\displaystyle
    \varphi &: \quad \{ (L,\tau,\sigma)  \mid  (L,\tau,\sigma) \in \mathcal{A}_n \times H_n \times H_n \text{ for some }n \in \mathbb{N}\} \longrightarrow \mathcal{V}\\
    & (L, \tau, \sigma) \mapsto (L_\tau,\sigma)_v, \quad \text{where } L \in \mathcal{A}_n \text{ and } \tau, \sigma \in H_n.  \\
\end{split}
\end{equation*}
\noindent where, $(L_\tau,\sigma)_v$ denotes the virtual closure, with respect to the closure permutation, $\sigma$, of the linkoid, $L$, with strand permutation, $\tau$.
\label{def-closure-map}
\end{definition}

\begin{figure}[ht!]
    \centering
\includegraphics[scale=0.5]{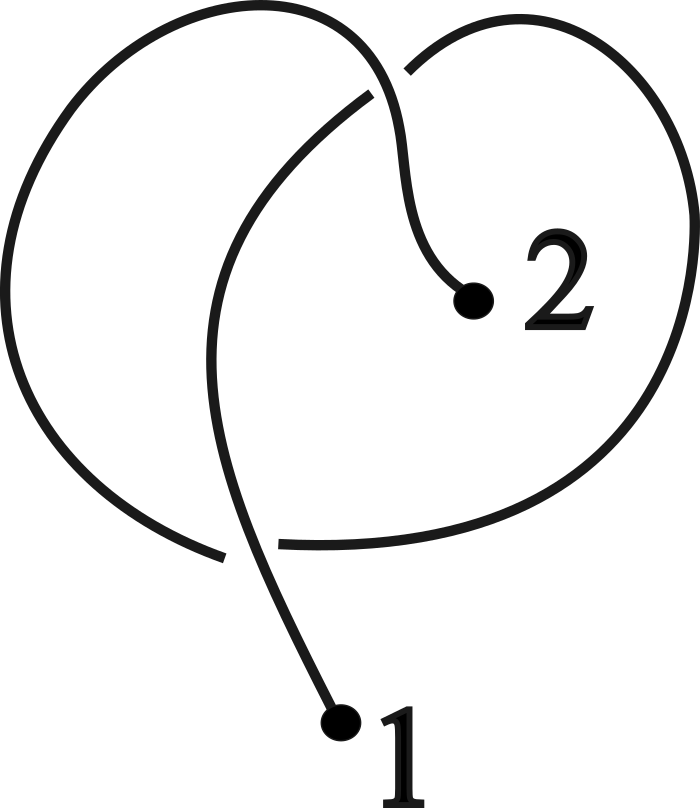}\hspace{50pt} \raisebox{0.1\totalheight}{\includegraphics[scale=0.5]{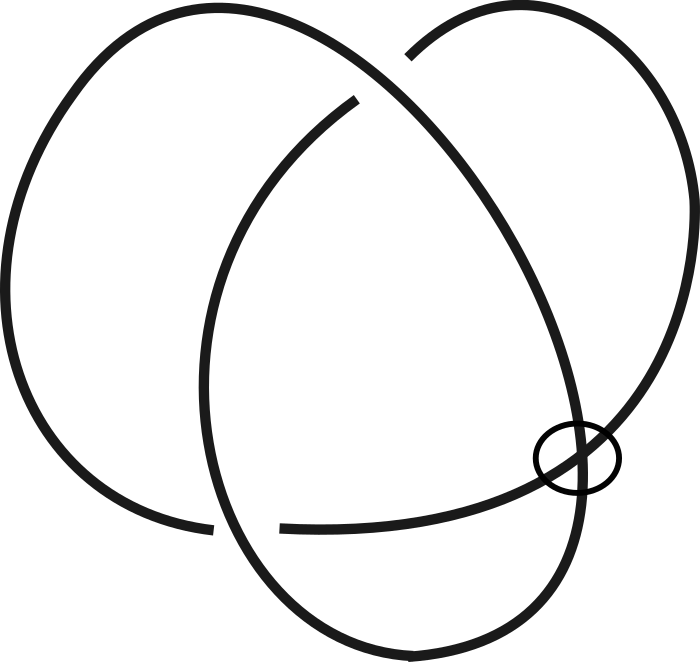}} \hspace{40pt} \includegraphics[scale=0.5]{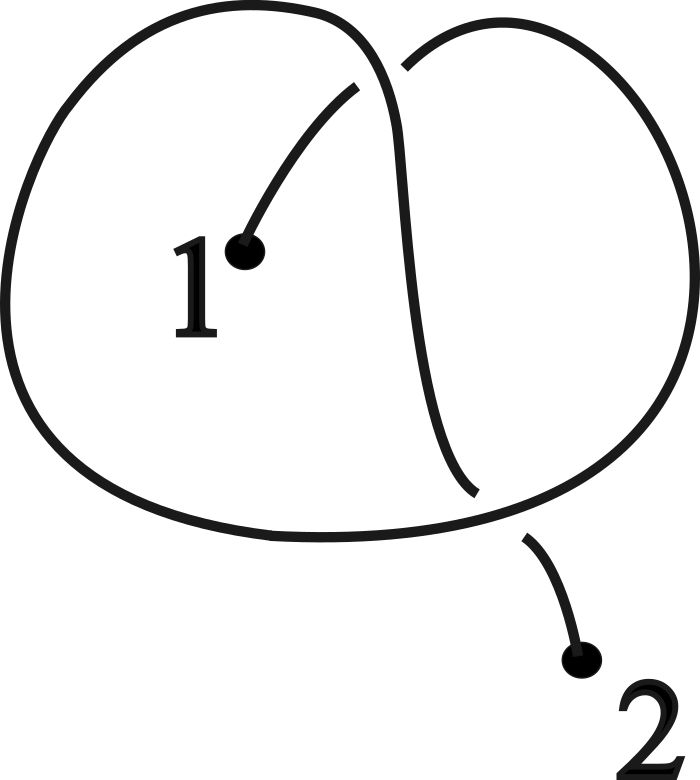} 
    \caption{(Left and Right) Two non-equivalent knotoid diagrams, $L_1$ and $M_1$, respectively. (Center) The virtual closures of $L_1$ and $M_1$ are equivalent and they correspond to the virtual trefoil knot.}
    \label{trefoil-closures-VM}
\end{figure}

\begin{theorem}\label{surj-thm}
The virtual closure map, $\varphi$, as given in Definition \ref{def-closure-map}, is surjective but not injective.
\end{theorem}
\begin{proof}
 Let $L_1$ and  $M_1$ be the two knotoids given in Figure \ref{trefoil-closures-VM} and for all $n \geq 2$, let, 
\begin{center}
  $\displaystyle L_n = L_1 \hspace{3pt} \bigsqcup_{i=2}^{n} \hspace{5pt} \raisebox{-13pt}{\includegraphics[scale=0.35]{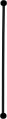}} _{2i}^{2i-1}$\hspace{10pt} and \hspace{10pt}$\displaystyle M_n = M_1 \hspace{3pt} \bigsqcup_{i=2}^{n} \hspace{3pt} \raisebox{-13pt}{\includegraphics[scale=0.35]{fig/triv-1-strand.png}} _{2i}^{2i-1}$  
\end{center}
 \noindent where $\raisebox{-11pt}{\includegraphics[scale=0.3]{fig/triv-1-strand.png}} _{2i}^{2i-1}$ is a crossingless open segment with endpoints $2i$ and $2i-1$, for all $2\leq i \leq n$. Clearly, $L_n$ and $M_n$ are not equivalent as linkoids. For both  $L_n, M_n \in \mathcal{A}_n$, let the strand permutation, $\tau$ and the closure permuataton $\sigma$ be each equal to $\prod_{i=1}^{n}(2i-1 \quad 2i)$. The map $\varphi$ is not injective since,
\begin{center}
    $\varphi(L_n,\tau,\sigma) = \raisebox{-13pt}{\includegraphics[scale=0.4]{fig/v-tref-closed.png}}  \hspace{3pt} \bigsqcup_{i=2}^{n} \hspace{3pt} \bigcirc = \varphi(M_n, \tau, \sigma).$
\end{center}

For any $V \in \mathcal{V}$, let $DV$ be a diagram with $m$ virtual crossings. For each virtual crossing, a unique arc may be chosen which passes through the virtual crossing but does not interact with the rest of $DV$. These $m$ arcs may be labelled as $(2i-1 \quad 2i)$, where $1\leq i \leq m$. Let $L$ be the linkoid diagram with $m$ components which is identical to $DV$ but does not contain the arcs bounded by $(2i-1 \quad 2i)$. Clearly, $DV$ is the virtual closure of $L$ with respect to $\sigma = \prod_{i=1}^{n}(2i-1 \quad 2i)$. Hence, $\varphi$ is surjective.

\end{proof}

\noindent The virtual closure map has the following properties :

\begin{enumerate}
    \item The image of the restriction map, ${\varphi}\big|_{L,\tau} : H_n \longrightarrow \mathcal{V}$, where $n$ is the number of components in $L$ is a subset of $\mathcal{V}$ and is denoted as $\displaystyle \operatorname{Im}({\varphi}{\big|}_{L,\tau})$. The number of elements in this subset is equal to $\displaystyle {\big|}\operatorname{Im}({\varphi}{\big|}_{L,\tau}){\big|} \leq |H_n| = \frac{1}{n!} \binom{2n}{2} \binom{2n-2}{2}\ldots \binom{2}{2}$. Therefore, there can be at most $|H_n|$ virtual knots/links that represent $L$ via the virtual closure map.
    \item Let $L$ be a trivial linkoid with $n$ components, and without loss of generality, let $\tau = \prod_{i=1}^{n}(2i-1 \quad 2i)$ i.e. $\displaystyle L=\bigsqcup_{i=1}^{n} \hspace{5pt} \raisebox{-11pt}{\includegraphics[scale=0.3]{fig/triv-1-strand.png}} _{2i}^{2i-1}$. Then, the virtual closure of $L$ with respect to any $\sigma \in H_n$ is a trivial link (up to generalized Reidemeister moves) with $\leq n$ components. In particular, virtual closure with respect to $\displaystyle \sigma =(1 \quad n) \prod_{i=1}^{n-1}(2i \quad 2i+1)$ corresponds to an unknot whereas,  virtual closure with respect to $\tau$ corresponds to a trivial link with $n$ components. Hence, $\displaystyle {\big|}\operatorname{Im}({\varphi}{\big|}_{L,\tau}){\big|} = n$. 
    \item Consider $\tau$ and $\sigma$ defined on $2n$ endpoints. For a given virtual link, $V \in \mathcal{V}$, of $m$ components, the preimage of the restriction map, $\displaystyle {\varphi}{\big|}_{\tau, \sigma} : \mathcal{A}_n \longrightarrow \mathcal{V}$, is the set of all $n$-component linkoids, with strand permutation $\tau$, whose virtual closure with respect to $\sigma$ is equal to $V$.
    \item  Without loss of generality, let $\tau = \prod_{i=1}^{n}(2i-1 \quad 2i)$, for a linkoid $L$ with $n$ components.  If $L$ has more than 1 component, ${\big|}\operatorname{Im}({\varphi}{\big|}_{L,\tau}){\big|}>1$  (Since virtual closure with respect to $\tau$ yields a virtual link with $n$-components, whereas  virtual closure with respect to $\sigma =(1 \quad n) \prod_{i=1}^{n-1}(2i \quad 2i+1)$ corresponds to a virtual knot with 1 component).
\end{enumerate}

\begin{remark}
    Under the map in Definition \ref{def-closure-map}, linkoids can be associated to embeddings of links in thickened surfaces of arbitrary genus.
\end{remark}

\section{Invariants of Linkoids via Virtual Closure}
\label{sec-inv-via-vc}

The framework introduced in the previous sections enables to define a new collection of linkoid invariants via invariants of virtual links. 

The following are some well-known invariants of virtual knots/links: fundamental group, quandles, biquandles, odd writhe, Jones polynomial, Khovanov homology, arrow polynomial, index polynomial, quantum link invariants and virtual Vassiliev invariants. All these can be defined for linkoids using the following theorem:

\begin{theorem}
Let $L_{\tau}$ be a linkoid, where $\tau$ is the strand permutation and let $\sigma$ a closure permutation. For any $\mathcal{F}$, an invariant of virtual links, let the map $\mathcal{F}_\sigma$, be defined as,
\begin{equation}\label{F_inv}
    \displaystyle {\mathcal{F}_\sigma}(L_\tau) : = \mathcal{F}((L_\tau,\sigma)_v),
\end{equation}
\noindent where $(L_\tau,\sigma)_v$ is the virtual closure of $L_\tau$ corresponding to  $\sigma$. Then $\mathcal{F}_\sigma$ is an invariant of linkoids.
\label{thm-F-inv}
\end{theorem}

\begin{proof}
Let $L_\tau$ and $K_\tau$ be diagrams that differ by a Reidemeister move on linkoids. By Remark \ref{rem-clos1}, the virtual closure diagrams, $(L_\tau,\sigma)_v$ and $(K_\tau, \sigma)_v$, represent the same virtual knot, where $\sigma$ is the given closure permutation for $\mathsf{L}$.
Thus, $$\displaystyle \mathcal{F}_\sigma(L_\tau) =  \mathcal{F}((L_\tau,\sigma)_v) = \mathcal{F}((K_\tau,\sigma)_v)=\mathcal{F}_\sigma(K_\tau),$$.
%Since $\mathcal{F}$ is an invariant of virtual knots,\quad  $\displaystyle \mathcal{F}((L_\tau,\sigma)_v) = \mathcal{F}((K_\tau,\sigma))$. By Equation \ref{F_inv},\quad $\displaystyle \mathcal{F}_\sigma(L_\tau) = \mathcal{F}((L_\tau,\sigma)_v)$ and $\displaystyle \mathcal{F}_\sigma(K_\tau) = \mathcal{F}((K_\tau,\sigma))$. Finally we have, 
%$$\displaystyle \mathcal{F}_\sigma(L_\tau) = \mathcal{F}_\sigma(K_\tau),$$
%thereby proving that, $\mathcal{F}_\sigma$ is an invariant for linkoids up to Reidemeister moves.
\end{proof}
\begin{remark}
    Note that the virtual class of the closure, with respect to a closure permutation $\sigma$, of a linkoid $\mathsf{L}_\tau$ is by itself an invariant of $\mathsf{L}_\tau$. This virtual class of the closure is denoted $(\mathsf{L}_\tau, \sigma)_v$. 
\end{remark}

Theorem \ref{thm-F-inv} gives rise to the following corollary regarding link-type linkoids, which plays a critical role in Section \ref{sec-open-curves}, which discusses entanglement of open curves in 3-space.

\begin{corollary}
    Let ${L}_\tau$ be a link-type linkoid of type $\mathcal{C}$ (See Definition \ref{def-link-type}), with respect to a closure permutation $\sigma$, where $\tau$ denotes the strand permutation. Let   $\mathcal{F}$ denote an invariant of virtual links and $\mathcal{F}_\sigma$ the corresponding invariant of linkoids. Then 
    \begin{equation}\label{F-linktype}
    \displaystyle \mathcal{F}_\sigma(L_\tau) = \mathcal{F}(\mathcal{C})
    \end{equation}
    %\noindent where, $\mathcal{C}=(L_\tau,\sigma)_v$ is the classical knot/link obtained from the closure of $L_\tau$ with respect to the closure permutation, $\sigma$. 
\end{corollary}

\begin{proof}
It follows by Definition \ref{def-link-type} and Theorem \ref{thm-F-inv}.%Since $L_\tau$ is a diagram of a link-type linkoid, $\mathsf{L}_\tau$, with respect to the closure permutation $\sigma$, therefore the virtual closure $\mathcal{C}=(L_\tau,\sigma)_v$ corresponds to a classical knot. By Theorem \ref{thm-F-inv}, $\displaystyle \mathcal{F}_\sigma(L_\tau) = \mathcal{F}((L_\tau,\sigma)_v) = \mathcal{F}(\mathcal{C})$, which makes $\mathcal{F}_\sigma$, an invariant for link-type linkoids, and $\mathcal{F}$, an invariant of classical links. 
\end{proof}

The virtual closure map, $\varphi$, defined in Definition \ref{def-closure-map} is not injective, where $n \in \mathbb{N}$. Therefore, linkoids which are mapped to the same virtual knot/link under the virtual closure map, return the same value for any invariant of linkoids as defined in Theorem \ref{thm-F-inv}.

A more detailed discussion on some specific invariants of linkoids is provided in the remainder of this section. A combinatorial approach to determine the Jones polynomial and the arrow polynomial of linkoids, with respect to arbitrary closure permutations, is discussed and it is shown to be consistent with the definition in terms of virtual closure. Moreover, the extension of invariants such as, height, genus, odd writhe and the affine-index polynomial is discussed for linkoids in terms of virtual closure.

\subsection{Height and Genus of Linkoids and Virtual Links}

A fundamental topological invariant of a knotoid (linkoid with one component) is its height \cite{Turaev2012, Gugumcu2017}. 
Here, a generalized definition of the height of linkoids is given by referring to virtual crossings and the choice of closure permutation. 

\begin{definition}\textit{(Height of a linkoid with respect to $\sigma$)}
The \textit{height of a linkoid diagram} $L_\tau$, corresponding to a closure permutation $\sigma$, is the number of virtual crossings in its minimal virtual closure. The \textit{height of a linkoid}, $\mathsf{L}_\tau$ corresponding to $\sigma$ is defined as the minimum of the diagrammatic heights, taken over all knotoid/linkoid diagrams equivalent to $\mathsf{L}_\tau$.
\label{def-gen-hgt}
\end{definition}

\begin{remark}
   The  height of a knotoid, as defined in \cite{Turaev2012, Gugumcu2017}, is recovered as a special case of Definition \ref{def-gen-hgt} by taking the strand closure of the knotoid. 
\end{remark}

Via virtual closure, the genus of a linkoid can be defined as follows:

\begin{definition}(\textit{Genus of linkoid with respect to $\sigma$} )
 The \textit{genus}, $g$, of a linkoid $L_\tau$ with respect to a closure permutation, $\sigma$ is defined as the genus of the minimal embedding surface of the virtual knot/link, $(L_\tau,\sigma)$.
\end{definition}

\subsection{The Jones polynomial of Linkoids and Virtual Links}

In this section, the Jones polynomial of a knotoid/linkoid,  with respect to a given closure permutation, is discussed as the normalized bracket polynomial of the knotoid/linkoid, with the substitution $A=t^{-1/4}$. The bracket polynomial for linkoids was recently defined combinatorially, without referring to its virtual closure  in \cite{Barkataki2022}. In this section, it is shown that the definition in \cite{Barkataki2022} coincides with that of the strand virtual closure. Moreover, it is shown that the combinatorial definition of the bracket polynomial can be generalized to account for arbitrary closure permutations and that it coincides with the corresponding definition in terms of virtual closure. 

Below follows a combinatorial definition of the Jones polynomial of linkoids :

\begin{definition}(\textit{Jones polynomial of a linkoid with respect to $\sigma$})
The Jones polynomial of an oriented linkoid diagram, $L$, with respect to a closure permutation, $\sigma$, is defined as, \begin{equation}f_{L^\sigma} = (-A^3) ^{-\mathsf{Wr}(L)} \langle L^\sigma \rangle \hspace{0.05cm},\end{equation}
where  $\mathsf{Wr}(L)$ is the writhe of the linkoid diagram and $\langle L^\sigma \rangle$ is the generalized bracket polynomial of $L$, with respect to $\sigma$ (See Definition \ref{def_bkt}).
\label{def_jones}
\end{definition}

The generalized bracket polynomial in Definition \ref{def_jones} can be defined combinatorially as follows:

\begin{definition}(\textit{Generalized bracket polynomial of a linkoid with respect to $\sigma$})
Let $L$ be a linkoid diagram with $n$ components, $E$ be the set of its labelled endpoints  and $\sigma$ be a closure permutation. The generalized bracket polynomial of $L$, with respect to $\sigma$, is completely characterised by the following skein relation and conditions:
\begin{equation}
%\begin{split}
\left\langle\raisebox{-7pt}{\includegraphics[width=.04\linewidth]{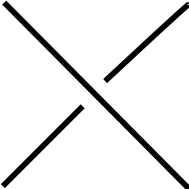}}\right\rangle=A\left\langle\raisebox{-7pt}{\includegraphics[width=.04\linewidth]{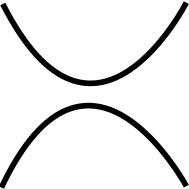}}\right\rangle +A^{-1}\left\langle\raisebox{-7pt}{\includegraphics[width=.04\linewidth]{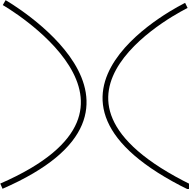}}\right\rangle, \hspace{0.5cm}\bigl\langle \tilde{L}\cup \bigcirc\bigr\rangle=d\bigl\langle \tilde{L}\bigr\rangle, \hspace{0.5cm} \left\langle S_\ell \right\rangle=d^{|E/\langle \tau_S, \sigma \rangle|}\quad ,
\label{bkt_sk_L}
%\end{split}
\end{equation}
\noindent where $d=-A^2-A^{-2}$ ; $\tilde{L}$ is any linkoid diagram ; $S_\ell$ is the trivial linkoid formed by the collection of $n$ labelled open segments in a state, $S$, in the state sum expansion of $L$ and $\tau_S$ is the strand permutation of $S_\ell$. The term,
$|E/\langle \tau_S, \sigma \rangle|$, is the number of distinct segment cycles in $S_\ell$,  with respect to $\sigma$.

The \textit{generalized bracket polynomial} of $L$, with respect to $\sigma$, can be expressed as the following state sum expression :
\begin{equation}
    \bigl\langle L^{\sigma} \bigr\rangle \hspace{0.05cm} := \hspace{0.05cm} \sum_S A^{\alpha (S)} d^{\operatorname{circ}(S) + |E/\langle \tau_S, \sigma \rangle| -1} \quad , 
\end{equation}
\noindent where $S$ is a state corresponding to a choice of smoothing over all double points in $L$;\hspace{0.05cm} $\alpha(S)$ is the algebraic sum of the smoothing labels of $S$ and $\operatorname{circ}(S)$ is the number of circular components in $S$. 
\label{def_bkt}
\end{definition}

\noindent The generalized bracket polynomial of linkoids has the following properties :
\begin{enumerate}
    \item It is a one-variable polynomial that preserves the underlying skein relation of the bracket polynomial of classical knots and knotoids. 
    \item If $L$ is a link diagram, with no open components, then the labelled set of endpoints, $E=\emptyset$ and the generalized bracket polynomial  coincides with the traditional Kauffman bracket polynomial of $L$ i.e. $\displaystyle\langle L \rangle = \sum_S A^{\alpha (S)} d^{\operatorname{circ}(S) -1}\hspace{0.05cm}.$
    \item If $L$ is a knotoid diagram (linkoid with one component), then the generalized bracket polynomial coincides with the Kauffman bracket polynomial of knotoids. Any closure permutation connects the 2 open endpoints of $L$, thus $|E/\langle \tau_S, \sigma \rangle|=1$ for all $S$. Thus, $\displaystyle \langle L \rangle  =  \sum_S A^{\alpha (S)} d^{\operatorname{circ}(S)}\hspace{0.05cm}.$
    \item For an $n$-component linkoid diagram,  $L$, with strand permutation, $\tau$, and no crossings, the state sum only has 1 state, $S$. In this case, $\tau_S=\tau$ and $\displaystyle \langle L^\sigma \rangle = d^{|E/\langle \tau, \sigma \rangle|-1}\hspace{0.05cm}$, where $E$ denotes the set of labelled endpoints of $L$.
    \item If a linkoid diagram, $L$, is of link-type, with respect to a closure permutation, $\sigma$, the bracket polynomial coincides with that of the corresponding link upon the closure of endpoints.
\end{enumerate}

According to  Theorem \ref{thm-F-inv}, the Jones polynomial of a linkoid, $L_\tau$, with respect to a closure permutation $\sigma$, can be defined as the Jones polynomial of its virtual closure, $(L_\tau,\sigma)_v$. Theorem \ref{jones-l-im} shows the equivalence between Definition \ref{def_jones} and the definition of the Jones polynomial via virtual closure.

\begin{theorem}
 Let $L$ be a linkoid diagram with strand permutation, $\tau$ and closure permutation, $\sigma$ on the labelled set of endpoints, $E$. 
The Jones polynomial of $L$, according to Definition \ref{def_jones}, is equal to the Jones polynomial of the virtual closure, $(L_\tau,\sigma)_v$, namely, $f_{L^\sigma}=f_{ (L_\tau,\sigma)_v}$.
\label{jones-l-im}
\end{theorem}
\begin{proof}
 $L$ and $(L_\tau,\sigma)_v$ have equal number of classical crossings, therefore 
 $ \mathsf{Wr}(L) = \mathsf{Wr}\left((L_\tau,\sigma)_v\right)$.
Let $S$ be a state in the state sum expansion of $L$ corresponding to a choice of smoothing over the classical crossings in $L$. By the one-to-one correspondence between the classical crossings in $L$ and $(L_\tau,\sigma)_v$, there is a state $S_v$ in $(L_\tau,\sigma)_v$ corresponding to the same choice of smoothing.  $S_v$ is the virtual closure of $S$ since,  $S=S_v \backslash \{r \mid r \text{ is a virtual closure arc on 
 } E \}$. By Corollary \ref{corr-seg-virconn}, $S$ and $S_v$ have the same number of components. Thus $  \bigl\langle S^\sigma \bigr\rangle  =   \bigl\langle S_v \bigr\rangle $ for all states $S$ of $L$. Hence, $ \bigl\langle L^{\sigma} \bigr\rangle =  \bigl\langle (L_\tau, \sigma)_v \bigr\rangle$. Therefore, the Jones polynomial of $L$,  according to Definition \ref{def_jones}, is equal to the Jones polynomial of its virtual closure, i.e. 
 \begin{equation}\displaystyle
    f_{L^\sigma} = (-A^3) ^{-\mathsf{Wr}(L)} \langle L^\sigma \rangle  =  (-A^3) ^{-\mathsf{Wr}\left((L_\tau,\sigma)_v\right)} \langle (L_\tau,\sigma)_v \rangle = f_{(L_\tau,\sigma)_v}. 
 \end{equation}
\end{proof}

\begin{remark}
Theorem \ref{jones-l-im} allows to compute the Jones polynomial of virtual knots/links by using linkoids and without using the closure arcs. 
\end{remark}

 \begin{example}\label{bkt-ex-linkoid}
    Let us consider the linkoid, $\displaystyle  L = \raisebox{-15pt}{\includegraphics[scale=0.5]{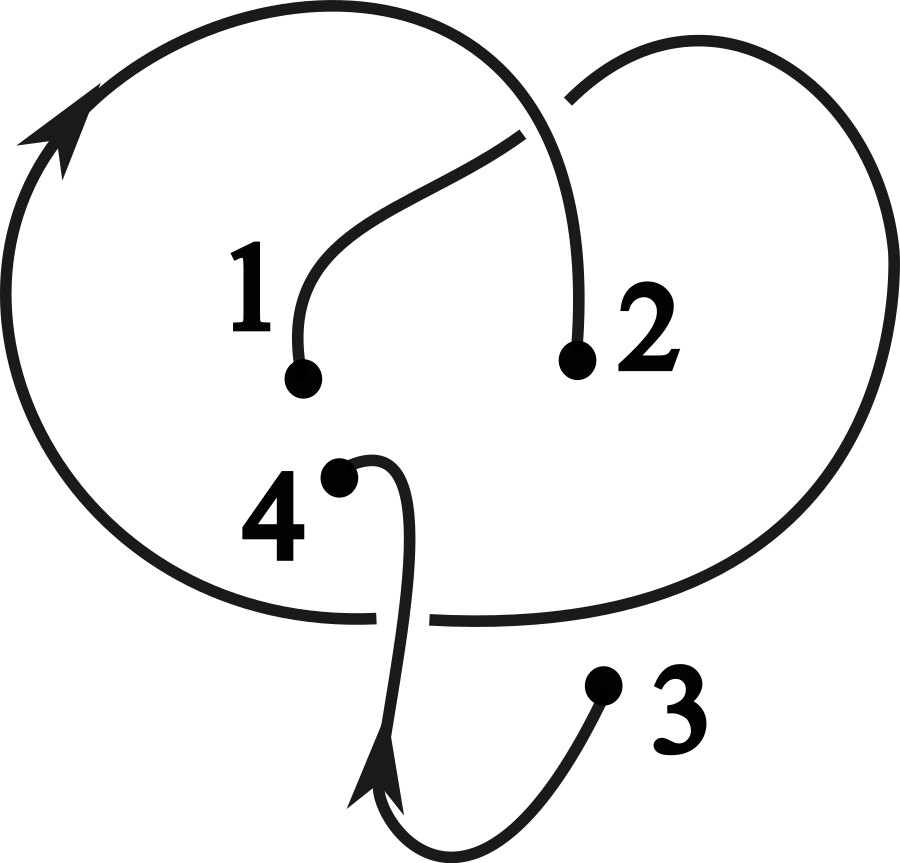}}$, with strand permutation $\tau=(1 \quad 2)(3 \quad 4)$.  The virtual closure of $L$ with respect to $\sigma = (1 \quad 4)(2 \quad 3)$ corresponds to the virtual trefoil knot, $\displaystyle  K = \raisebox{-12pt}{\includegraphics[scale=0.5]{fig/v-tref-closed.png}}$. Hence, the Jones polynomial of $L$ with respect to $\sigma$ is given as, $f_{L^\sigma} = f_{K} = (-A^3)^{-2}(A^2-A^{-4}+1) = A^{-4}- A^{-10}+ A^6$.
\end{example}

\subsection{The Arrow polynomial of Linkoids and Virtual Links}

The construction of the arrow polynomial invariant begins with the oriented state summation of the bracket polynomial where each local smoothing is either an oriented smoothing or a disoriented smoothing, as shown in Equation \ref{arrow_sk_L}, for both positive and negative classical crossings in a  diagram \cite{dye2009virtual}.

\begin{equation}
\begin{split}
\left\langle\raisebox{-7pt}{\includegraphics[width=.045\linewidth]{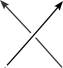}}\right\rangle_{\mathcal{A}}=A\left\langle\raisebox{-7pt}{\includegraphics[width=.045\linewidth]{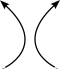}}\right\rangle_{\mathcal{A}} +A^{-1}\left\langle\raisebox{-7pt}{\includegraphics[width=.04\linewidth]{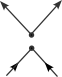}}\right\rangle_{\mathcal{A}}, &\hspace{0.5cm} \left\langle\raisebox{-7pt}{\includegraphics[width=.045\linewidth]{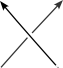}}\right\rangle_{\mathcal{A}}=A^{-1}\left\langle\raisebox{-7pt}{\includegraphics[width=.045\linewidth]{fig/ori-smooth.png}}\right\rangle_{\mathcal{A}} +A\left\langle\raisebox{-7pt}{\includegraphics[width=.04\linewidth]{fig/disori-smooth.png}}\right\rangle_{\mathcal{A}}, \\ \hspace{0.5cm}\bigl\langle \tilde{L}\cup \bigcirc\bigr\rangle_{\mathcal{A}}=d\bigl\langle \tilde{L}\bigr\rangle_{\mathcal{A}},& \hspace{0.5cm} d= -A^2 - A^{-2},
\label{arrow_sk_L}
\end{split}
\end{equation}

For linkoid diagrams, any final state in the state sum expansion consists of  circular loops and/or segment cycles determined by the choice of closure permutation. A loop or a segment cycle, which is obtained by disoriented smoothings, also includes one or more paired cusps. Each cusp can be regarded as an indication of one side of a given arc. The equivalence relation for reduction of cusps is given in Figure \ref{cancel}. This means that consecutive indications of both sides of an arc do not cancel, but consecutive indications of one side do cancel. With this reduction, the resulting state sum is invariant under the Reidemeister and virtual moves. 
\begin{figure}[ht!]
    \centering
    \includegraphics[scale=0.2]{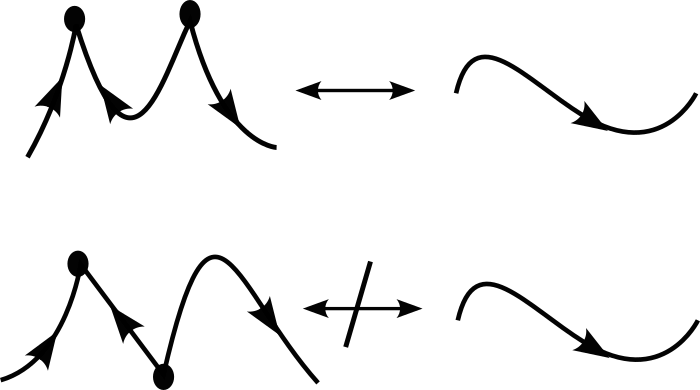}
    \caption{Equivalence moves on cusps in a state of the arrow polynomial.}
    \label{cancel}
\end{figure}
The generalized arrow polynomial of a linkoid can be defined as follows:

%As in \cite{dye2009virtual}, the cusps in a state component can be represent by nodal arrows and the values of 0 and 1 can be alternately assigned to each edge in the diagram to assign a sign to each node, as shown in Equation \ref{arrow_signs} :
%\begin{equation}
%5\raisebox{-7pt}{\includegraphics[width=.05\linewidth]{fig/disori-smooth.png}} \longrightarrow \raisebox{-7pt}{\includegraphics[width=.05\linewidth]{fig/nodal-arrow.png}},\hspace{0.5cm} \text{and} \hspace{0.5cm}
%\raisebox{-15pt}{\includegraphics[width=.2\linewidth]{fig/nodal-arrow-sign.png}}
%\label{arrow_signs}
%\end{equation}
%The nodal signs give the arrow number for each component in a state, which is related to the number of surviving cusps in a state component. Thus, the generalized arrow polynomial of a linkoid can be defined combinatorially as follows:

\begin{definition}(\textit{Generalized arrow polynomial of a linkoid with respect to $\sigma$})
Let $L$ be a linkoid diagram with $n$ components, $E$ be the set of its labelled endpoints  and $\sigma$ be a closure permutation. The \textit{generalized arrow polynomial} of $L$, with respect to $\sigma$, can be expressed as the following state sum expression :
\begin{equation}
    \bigl\langle L^{\sigma} \bigr\rangle_{\mathcal{A}} \hspace{0.05cm} := \hspace{0.05cm} \sum_S A^{\alpha (S)} d^{\operatorname{circ}(S) + |E/\langle \tau_S, \sigma \rangle| -1} \mathcal{K}_S \quad , 
\end{equation}
\noindent where $S$ is a state corresponding to a choice of smoothing (oriented or disoriented) over all classical in $L$;\hspace{0.05cm} $\alpha(S)$ is the algebraic sum of the smoothing labels of $S$ and $\operatorname{circ}(S)$ is the total number of circular components in $S$ and $|E/\langle \tau_S, \sigma \rangle|$ is the total number of segment cycles in $S$. Let $S_{i_1}, S_{i_2}, \ldots S_{i_m}$, with $i_1, i_2 \ldots i_m$ pairs of surviving cusps respectively, be all the  state components in $S$ with surviving pairs of cusps. Then $\langle S_j \rangle_{\mathcal{A}} = K_j$, for all $j \in \{ i_1, i_2, \ldots i_m \}$ and $\mathcal{K} = \prod_{j=1}^m K_{i_j} $ (See Figure \ref{Ki}).
\label{def_arrow-combi}
\end{definition}

\begin{figure}[ht!]
    \centering
    \includegraphics[scale=0.3]{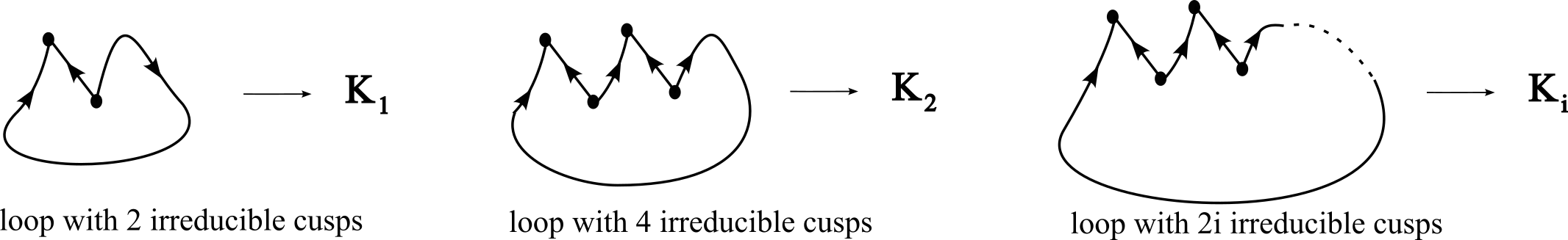}
    \caption{The arrow polynomial of a loop with $2i$ irreducible cusps is denoted by the variable, $K_i$, where $i \in \mathbb{N}$.}
    \label{Ki}
\end{figure}

\begin{remark}
    The generalized arrow polynomial of a linkoid is a polynomial in variables $A, K_i$, where $i \in \mathbb{N}$ with integer coefficients.
\end{remark}

\begin{remark}
    In \cite{Gugumcu2017}, the arrow polynomial of knotoids was defined where a state in the state sum expansion can have a cusped long segment. If the long segment has $2i$ irreducible cusps, then the arrow polynomial of the long segment is denoted by the variable $\Lambda_i$.  
\end{remark}

According to  Theorem \ref{thm-F-inv}, the generalized arrow polynomial of a linkoid, $L_\tau$, with respect to a closure permutation $\sigma$, can be defined as the arrow polynomial of its virtual closure, $(L_\tau,\sigma)_v$. This defintion is equivalent to Defintion \ref{def_arrow-combi} due to the following therorem:

\begin{theorem} Let $L$ be a linkoid diagram with strand permutation, $\tau$ and $\sigma$ be a closure permutation. The arrow polynomial of $L_\tau$, as in Definition \ref{def_arrow-combi}, coincides with the arrow polynomial of its virtual closure, $(L_\tau,\sigma)_v$. 
\label{def-arr-linkoid}
\end{theorem}

\begin{proof}
    It follows similarly to Theorem \ref{jones-l-im}.
\end{proof}

Example \ref{arrow-states} shows the arrow polynomial state sum expansion for a particular linkoid. One of the final states in this example corresponds to a loop with cusps and virtual crossings. Via detour move, it can be seen that this state is in fact a loop with irreducible cusps (See Figure \ref{Ki}).

\begin{example}\label{arrow-states}
The arrow polynomial of \raisebox{-15pt}{\includegraphics[scale=0.2]{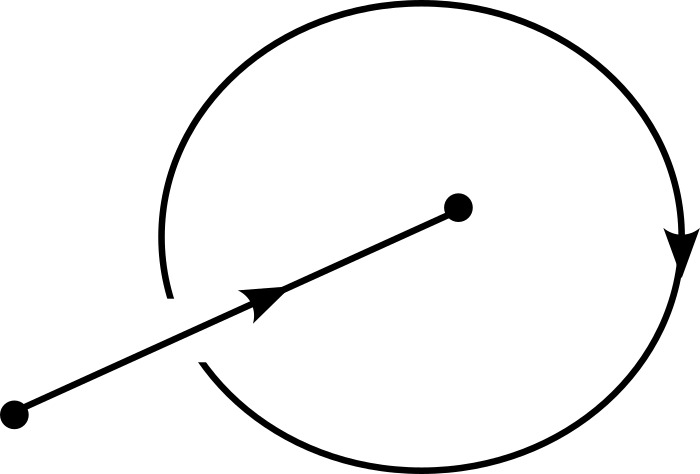}} is evaluated as follows :
    \begin{equation}
\begin{split}
\left\langle\raisebox{-15pt}{\includegraphics[width=.12\linewidth]{fig/CL-main.png}}\right\rangle_{\mathcal{A}}&= \left\langle\raisebox{-18pt}{\includegraphics[width=.1\linewidth]{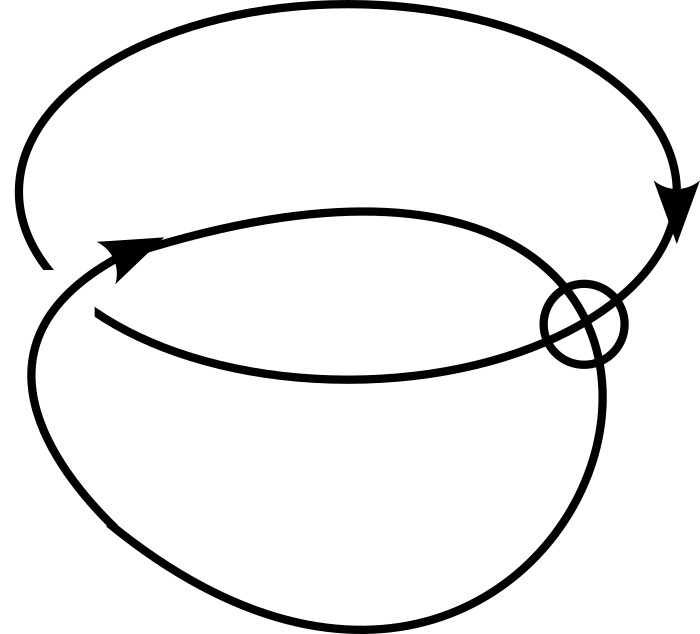}}\right\rangle_{\mathcal{A}}=A\left\langle\raisebox{-15pt}{\includegraphics[width=.1\linewidth]{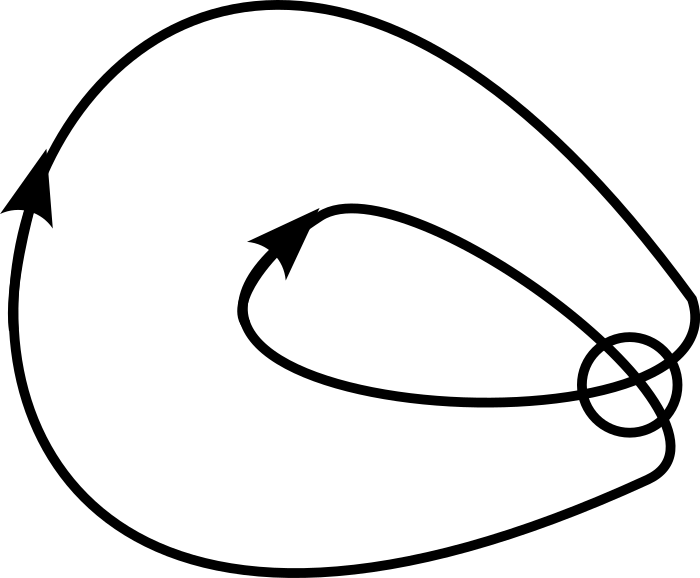}}\right\rangle_{\mathcal{A}} +A^{-1}\left\langle\raisebox{-20pt}{\includegraphics[width=.1\linewidth]{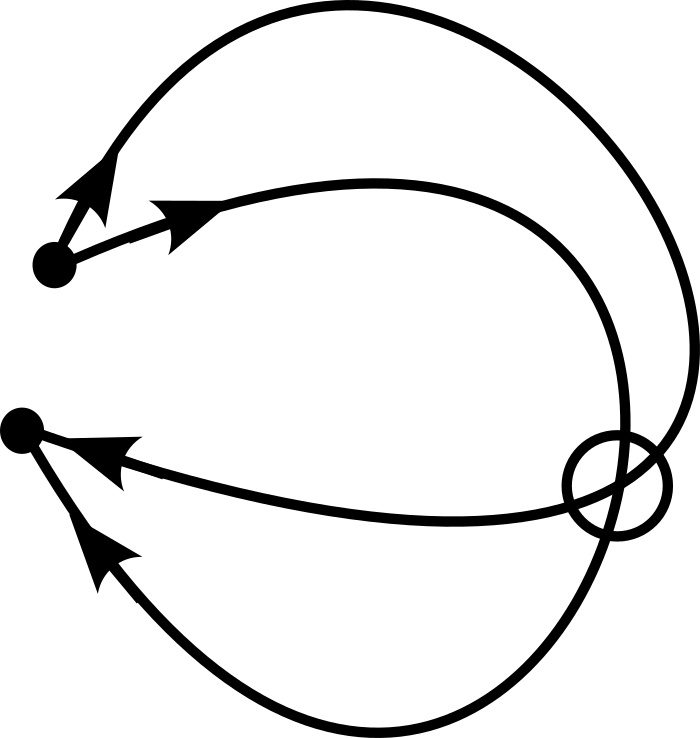}}\right\rangle_{\mathcal{A}}\\&=A +A^{-1}\left\langle\raisebox{-20pt}{\includegraphics[width=.1\linewidth]{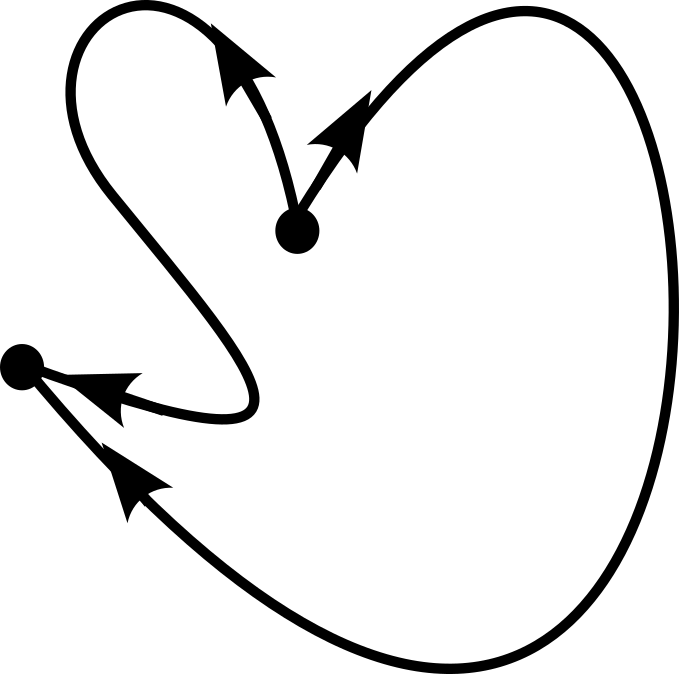}}\right\rangle_{\mathcal{A}} = A + A^{-1}K_1.\\
\label{arrow_simple}
\end{split}
\end{equation}
\end{example}

Example \ref{kishino-linkoid} illustrates the evaluatation the arrow polynomial of a 2 component linkoid:

\begin{figure}[ht!]
    \centering
\raisebox{-12pt}{\includegraphics[scale=0.35]{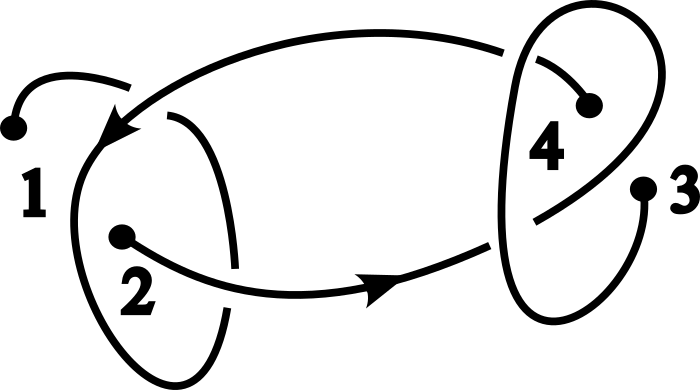}}\hspace{50pt}  \raisebox{-10pt}{\includegraphics[scale=0.32]{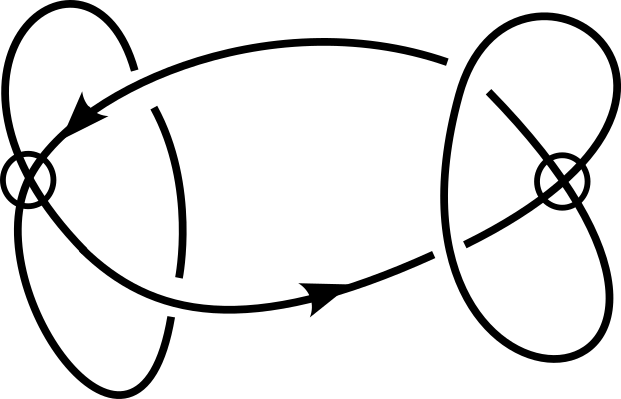}}
    \caption{(Left) A linkoid $L$ with 2 components and strand permutation, $\tau = (1 \quad 4)(2 \quad 3)$. (Right) The virtual closure of $L$, with respect to the closure permutation, $\sigma = (1 \quad 2)(3 \quad 4)$, is equal to the kishino knot, $K$.}
    \label{kishino-arrow}
\end{figure}

\begin{example}\label{kishino-linkoid}
    Let us consider the linkoid, $ L $, with strand permutation $\tau=(1 \quad 4)(2 \quad 3)$ and its   virtual closure, with respect to $\sigma = (1 \quad 2)(3 \quad 4)$, $K$ as in Figure \ref{kishino-arrow}. Since, the Kishino knot has trivial Jones polynomial, the Jones polynomial of $L$ with respect to $\sigma$ is also trivial (despite $L$ being a non-trivial linkoid). However, the arrow polynomial of $L$ is easily seen to be non-trivial, i.e. it matches the arrow polynomial of $K$, namely, $A^4 + 1 + A^{-4} - (A^4 + 2 + A^{-4})K^2_1 + 2K_2.$ Thus, with respect to $\sigma$, $L$ is a linkoid that the Jones polynomial cannot detect but the arrow polynomial is able to detect.
\end{example}

\subsection{The Affine-index polynomial of Linkoids and Virtual Links}

Given a labelled flat (without over/under information at crossings) virtual knot diagram, two numbers,  $W_{-}(c)$ and $W_{+}(c)$, can be defined at each classical node $c$. If there is a labeled classical node with left incoming arc $a$ and right incoming arc $b$, as in \raisebox{-9pt}{\includegraphics[scale=0.4]{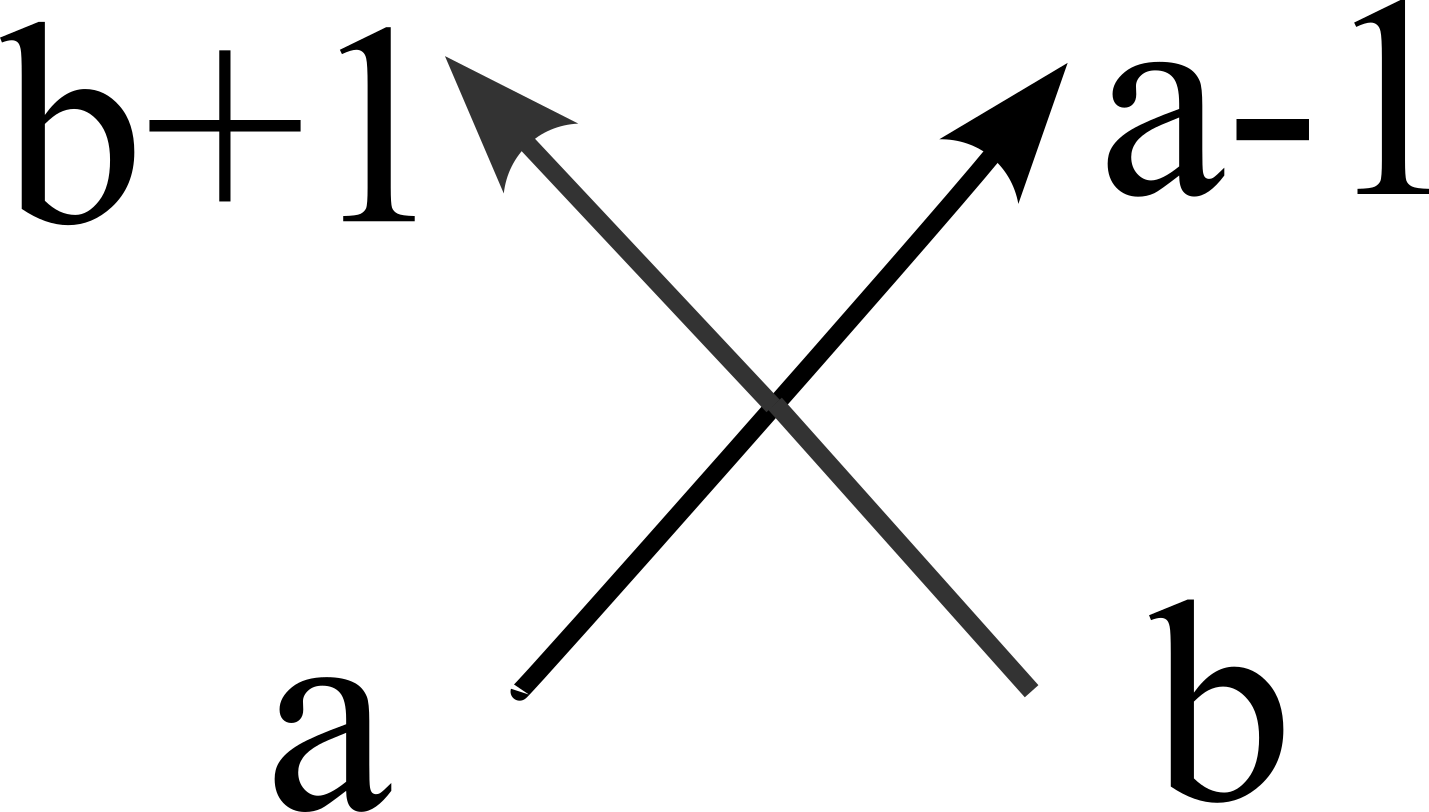}} , then the right outgoing arc is labeled $a-1$ and the left outgoing arc is labeled $b+1$. Then $W_{+}(c) = a-b-1 $ and $W_{-}(c) = -a + b + 1$ , respectively.
Note that $W_{-}(c)=-W_{+}(c)$ in all cases. Whereas, if there is a virtual node  with left incoming arc $a$ and right incoming arc $b$, then the right outgoing arc is labeled $a$ and the left outgoing arc is labeled $b$ (i.e. No change at a virtual crossing). The affine-index polynomial of a virtual knot was defined in \cite{Kauffman2013affine}, in terms of such integer labelling on flat virtual knot diagrams, in the following way: 

\begin{definition}(\textit{Affine index polynomial of virtual knots})
 The Affine Index Polynomial of a virtual knot diagram, $K$, is defined by the equation
 \begin{equation}
 P_K :=\sum_c \operatorname{sgn}(c)\left(t^{W_K(c)}-1\right)=\sum_c \operatorname{sgn}(c) t^{W_K(c)}-\mathsf{w r}(K) \quad,
\end{equation}
where $c$ denotes a classical crossing in $K$, $\operatorname{sgn}(c)$ is the sign of $c$, $\mathsf{wr}(K)$ is the writhe of $K$ and  $
W_K(c)=W_{\operatorname{sgn}(c)}(c)
$, obtained from labelling of arcs in the flat counterpart of $K$.
\end{definition}

According to  Theorem \ref{thm-F-inv}, the affine-index polynomial of a linkoid, $L_\tau$, with respect to a closure permutation $\sigma$, can be defined as the affine-index polynomial of its virtual closure, $(L_\tau,\sigma)_v$ provided $(L_\tau,\sigma)_v$ consists of only one component. 

\begin{definition}(\textit{Affine-index polynomial of a linkoid with respect to $\sigma$})
Let $L$ be a linkoid diagram with $n$ components, $E$ be the set of its labelled endpoints  and $\sigma$ be a closure permutation such that $(L_\tau,\sigma)_v$ consists of only one component. The \textit{affine-index polynomial} of $L$, with respect to $\sigma$, is defined as:
\begin{equation}
  P_{L^\sigma} := P_{(L_\tau,\sigma)_v} \quad, 
\end{equation}
\noindent where $P_{(L_\tau,\sigma)_v}$ is the affine -index polynomial of the virtual knot, $(L_\tau,\sigma)_v$.
\label{def_affine-linkoid}
\end{definition}

Example \ref{affine-example} illustrates the evaluatation the affine-index polynomial of a 2 component linkoid:

\begin{example}
Consider the linkoid, $\displaystyle  L = \raisebox{-15pt}{\includegraphics[scale=0.4]{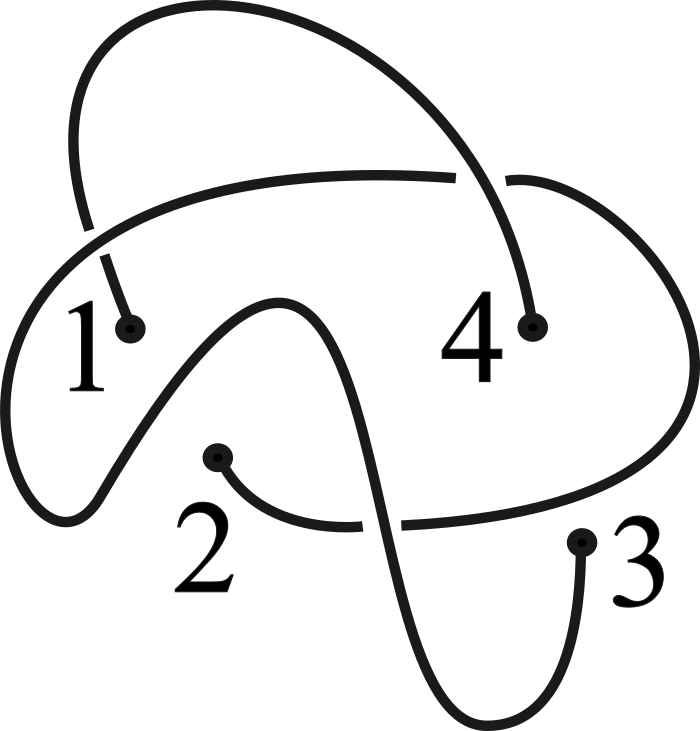}}$, with 2 components. The virtual closure of $L$, with respect to $\sigma = (1 \quad 2)(3 \quad 4)$, has only 1 component whose flat diagram is given as, $\displaystyle  FV = \raisebox{-15pt}{\includegraphics[scale=0.4]{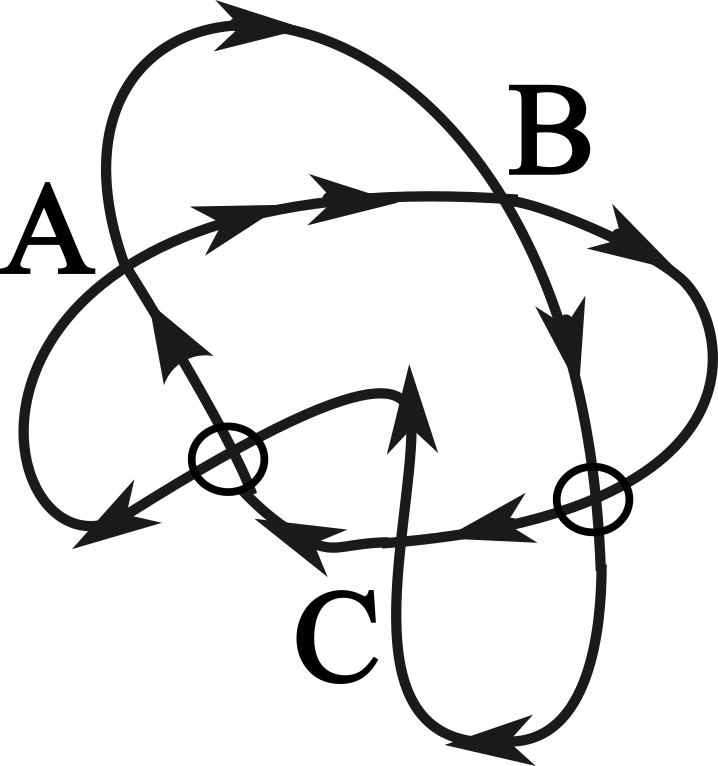}}$. The labels $A, B$ and $C$ correspond to classical crossings, which all have positive signs. Thus,  $W_{+}(A) = -2, W_{+}(B) = 2 $ and $W_{+}(C) = 0$ are respective the weights associated with the classical crossings. The virtual knot, $(L_\tau,\sigma)_v$ is one with unit Jones polynomial. However, it has a non-trivial affine-index polynomial, namely, 
$
P_{(L_\tau,\sigma)_v}=t^{-2}+t^2-2$. In other words, the linkoid $L$  has trivial Jones polynomial but non-trivial affine-index polynomial, with respect to $\sigma = (1 \quad 2)(3 \quad 4)$.
\label{affine-example}
\end{example}

\subsection{The Odd Writhe of Linkoids and Virtual Knots}

The odd writhe of a virtual knot, which was shown to be an invariant of virtual
knots in \cite{Kauffman2004b}, is defined as follows:

\begin{definition}(\textit{Odd crossing} and \textit{Odd writhe of a virtual knot})
A crossing of a virtual knot is called \textit{odd} if its crossing labels in the Gauss diagram have an odd number of crossing labels between them. The \textit{odd writhe} of the virtual knot is defined as the sum of the signs of the odd crossings. 
\label{def-odd-cr-wr}
\end{definition}

Note that in classical knots, the odd writhe is always zero, since all crossings of classical knots are
even. As an example, the virtual closure of the linkoid in Figure \ref{odd-wr-gauss} has odd writhe equal to $-1$, which proves that this knot is not equivalent to any classical knot.

\begin{figure}[ht!]
    \centering
    \includegraphics[scale=0.4]{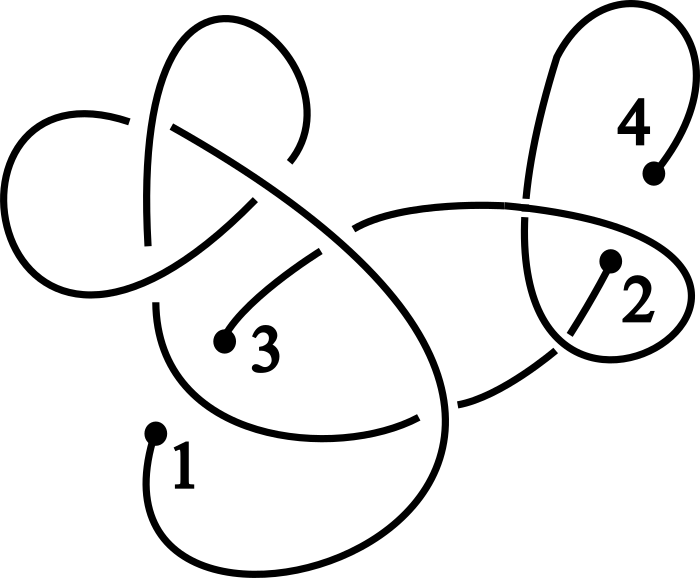}  \quad \includegraphics[scale=0.4]{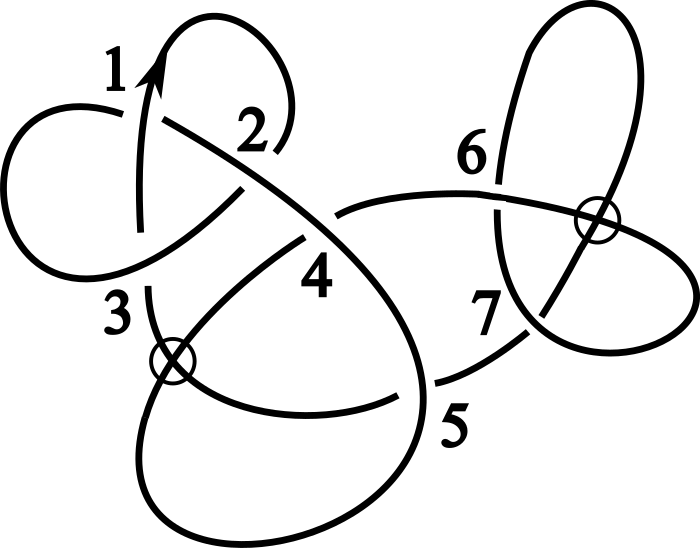}  \quad \includegraphics[scale=0.5]{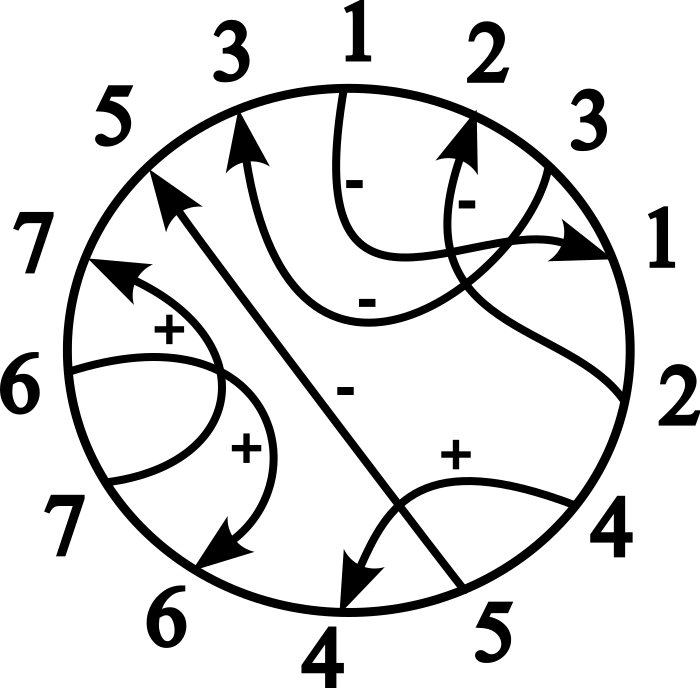}
    \caption{A linkoid diagram with strand permutation $\tau =(1 \quad 2)(3 \quad 4)$, its virtual closure with respect to $\sigma = (1 \quad 3)(2 \quad 4)$ and its Gauss diagram. The crossings labelled, $4, 5, 6$ and $7$ are odd and the odd-writhe of the virtual knot diagram is $-1$.}
    \label{odd-wr-gauss}
\end{figure}

Due to Theorem \ref{thm-F-inv}, the odd writhe of linkoids is well-defined  with respect to closure permutations which gives rise to a virtual knot (one component):

\begin{definition}(\textit{Odd writhe of a linkoid with respect to $\sigma$})
Let $L$ be a linkoid diagram and $\sigma$ be a closure permutation such that, $(L, \sigma)_v$ has one component. The \textit{odd writhe}, $\mathsf{owr}$ of $L$, with respect to $\sigma$, is defined as the sum of the signs of the odd crossings in $(L, \sigma)_v$.
\label{def-odd-sig-wr}
\end{definition}

\begin{remark}
For linkoids, the odd-writhe is an integer valued invariant.    
\end{remark}

\begin{remark}
    In the context of virtual links (2 or more components), the notion of parity does not extend naturally. Depending upon how the link components are examined, there is an ambiguity regarding whether a crossing is odd or even. This ambiguity can be remedied  by defining even and odd for crossings in individual components while labeling crossings shared by 2 components as link crossings \cite{kaestner2012}. 
\end{remark}

\section{The Virtual Spectrum of a Linkoid}
\label{sec-vir-spec}

 This section introduces the virtual spectrum of linkoids which leads to new invariants of linkoids that are not those of any given virtual closure. For any linkoid with fixed strand permutation, the virtual spectrum is the set of virtual links resulting from virtual closure over all possible closure permutations. More formally, it can be defined as follows:

\begin{definition}(\textit{Virtual spectrum of a linkoid})
The \textit{virtual spectrum} of a linkoid, $L$, with $n$ components and strand permutation, $\tau$, is defined to be the set, $\operatorname{spec}(L) := \displaystyle \{ \varphi_n(L,\tau, \sigma) : \sigma \in H_n \}$, where $\varphi_n$ is the virtual closure map defined on linkoids with $n$ components, $\displaystyle E=\{ 1, 2, \ldots , 2n\}$ and $\displaystyle H_n = \{ \sigma \in S_{2n} : \sigma^2(i) = i, \sigma(i) \neq i , \quad \forall i \in E \}$.
In other words, the virtual spectrum of $L$ is given by
$\displaystyle \operatorname{Im}({\varphi_n}{\big|}_{L,\tau})$.
\label{def-vir-spec}
\end{definition}

\begin{theorem}\label{thm-spec-is-inv}
For a linkoid, $L$, its virtual spectrum, $\operatorname{spec}(L)$ is an invariant of $L$.
\end{theorem}
\begin{proof}
 By Definition \ref{def-vir-spec}, each element in $\operatorname{spec}(L)$ is of the form, $(L,\sigma)_v$, where $\sigma$ is a closure permutation. For any $\sigma$, the virtual closure,  $(L,\sigma)_v$, is a linkoid invariant by  
 Remark \ref{rem-clos1} and Proposition \ref{prop-L-sig-iso}. Hence, $\operatorname{spec}(L)$ is an invariant of $L$. 
\end{proof}

\begin{corollary}\label{corr-equiv-diag-spec}
If two linkoid diagrams, $L_1$ and $L_2$, are equivalent up to Reidemeister moves, then $\operatorname{spec}(L_1) = \operatorname{spec}(L_2)$. 
\end{corollary}
\begin{proof}
    It follows from Theorem \ref{thm-spec-is-inv}
\end{proof}

\begin{corollary}
Let  $L_1$ and $L_2$ be two linkoid diagrams. If \quad $ \operatorname{spec}(L_1) \neq \operatorname{spec}(L_2)$, then $L_1$ and $L_2$ are distinct.
\end{corollary}
\begin{proof}
    It follows from Theorem \ref{thm-spec-is-inv}.
\end{proof}

\begin{remark}
    Two non-equivalent linkoids will have the same virtual spectrum only if all the closure permutations give the same virtual knots/links.
\end{remark}

\begin{example}
Note that, $\displaystyle \raisebox{-12pt}{\includegraphics[scale=0.3]{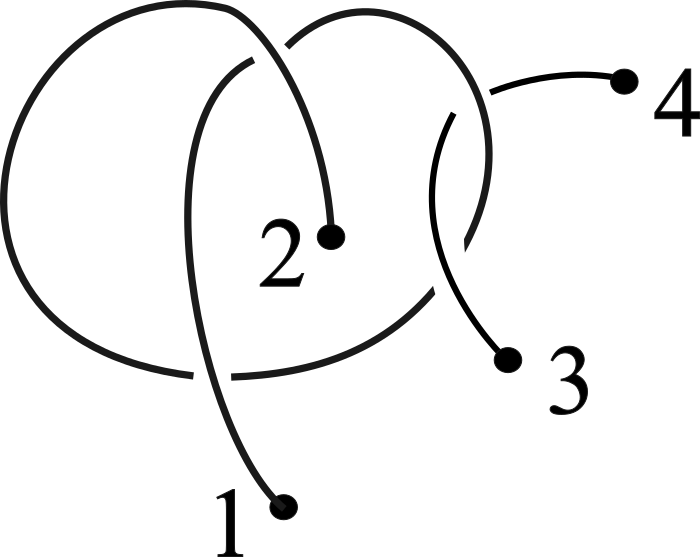}}$  and $\displaystyle \raisebox{-12pt}{\includegraphics[scale=0.3]{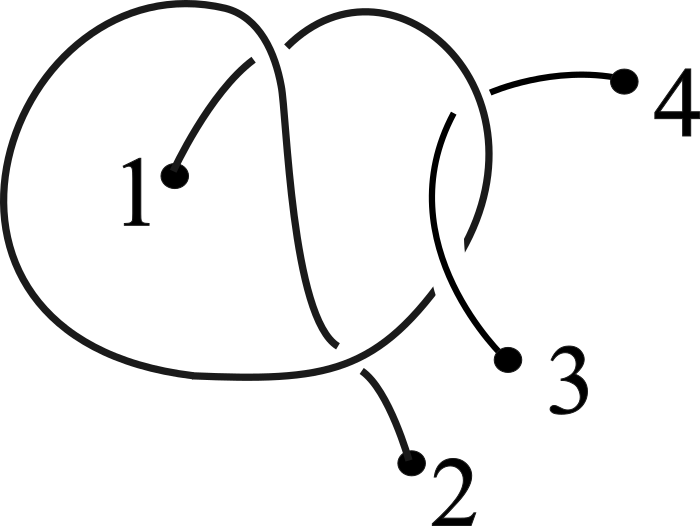}}$ are  two non-equivalent linkoids with same number of components and same virtual spectrum. 
\end{example}

\begin{remark}
    Property ii of the virtual closure map implies that the virtual spectrum of a trivial linkoid with $n$-components has size $n$.
\end{remark}

%By Definition \ref{def-link-type}, it can be determined when a linkoid is said to be of link-type $\mathcal{C}$ with respect to a given closure permutation. In the following, the virtual spectrum of a linkoid  is used to give the most general definition of a link-type linkoid.

%\begin{definition}(\textit{Spectral link-type linkoid})
%A linkoid $L$ is said to be \textit{spectral link-type} if there exists a classical knot/link in its virtual spectrum, $\operatorname{spec}(L)$. In other words, there exists a closure permutation, $\sigma \in S_{2n}$, such that, the virtual crossing number of the virtual closure, $(L,\sigma)_v$, is equal to zero. The linkoid, $L$, is said to be of spectral link-type $\mathcal{C}$ with respect to $\sigma$, when $(L,\sigma)_v$ is a classical knot/link of type $\mathcal{C}$.
%\label{def-link-type-gen}
%\end{definition}

\subsection{New Invariants of Linkoids from the virtual spectrum}

New invariants of linkoids can be defined, independent of any particular closure permutation, by means of the virtual spectrum of linkoids. 

\begin{definition}(\textit{Average spectral $\mathcal{F}$-invariant of a linkoid})
Let $\mathcal{F}$ be an invariant of virtual links, $L$ be a linkoid and $\displaystyle \mathcal{F}(\operatorname{spec}(L))$ be the collection of values of $\mathcal{F}$ over all virtual links in $\operatorname{spec}(L)$. The \textit{average spectral $\mathcal{F}$-invariant} of $L$ is defined as the average of $\displaystyle \mathcal{F}(\operatorname{spec}(L))$, namely :
\begin{equation*}\displaystyle
   \operatorname{avg}_{\mathcal{F}}(L) := \displaystyle \overline{ \mathcal{F}(\operatorname{spec}(L))} 
 = \frac{1}{|\operatorname{spec}(L)|} \sum_{\ell \in \operatorname{spec}(L)} 
    \mathcal{F}(\ell).
\end{equation*} 
\label{def-new-F-inv-linkoids}
\end{definition}

\noindent The invariance of the average spectral $\mathcal{F}$-invariant of a linkoid is given by the following theorem:

\begin{theorem}
Let $L$ be a linkoid, $\operatorname{spec}(L)$ be its virtual spectrum and let $\mathcal{F}$ be an invariant of virtual links. Let $\displaystyle \mathcal{F}(\operatorname{spec}(L))$ be the collection of values of $\mathcal{F}$ over all virtual links in $\operatorname{spec}(L)$. Then, the set, $\displaystyle \mathcal{F}(\operatorname{spec}(L))$ and the measure, $  \operatorname{avg}_{\mathcal{F}}(L)$, are invariants of the linkoid, $L$. 
\label{thm-new-F-inv}
\end{theorem}
\begin{proof}
 Since, $\mathcal{F}$ is an invariant of virtual links, there exits an invariant of linkoids, $\displaystyle {\mathcal{F}_\sigma}$, which is defined as $\mathcal{F}$ acting on the closure of the linkoid with respect to a closure permutation, $\sigma$ ( by Theorem \ref{thm-F-inv}). Therefore, the elements of the set, $\displaystyle \mathcal{F}(\operatorname{spec}(L))$, are of the form, $\displaystyle {\mathcal{F}_\sigma}(L)$, where $\sigma$ is a closure permutation. Each $\mathcal{F}_\sigma$ is an invariant of $L$, hence their collection, $\displaystyle \mathcal{F}(\operatorname{spec}(L))$ and the measure, $\operatorname{avg}_{\mathcal{F}}(L)$, are also invariants of $L$. 
\end{proof}

In this way, we can define the new invariants of linkoids such as, average spectral height, average spectral genus, average spectral Jones polynomial and average spectral arrow polynomial. The average spectral height and genus of linkoids are rational numbers whereas, the average spectral Jones polynomial and arrow polynomial of linkoids are polynomials with rational coefficients.

\begin{example}
For the linkoid, \raisebox{-7pt}{\includegraphics[scale=0.05]{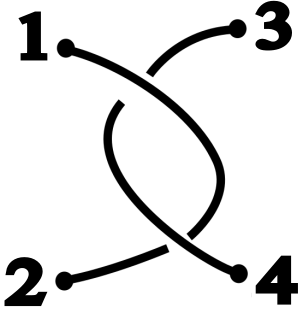}} with $\tau =(1 \quad 2)(3 \quad 4)$, the virtual spectrum is given by accounting for all $\sigma \in \{ (1 \quad 2)(3 \quad 4), (1 \quad 3)(2 \quad 4), (1 \quad 4)(2 \quad 3)\}$, i.e. 
$\operatorname{spec}\left(  \raisebox{-7pt}{\includegraphics[scale=0.05]{fig/hopf-open.png}} \right) = \left\{ \raisebox{-7pt}{ \includegraphics[scale=0.07]{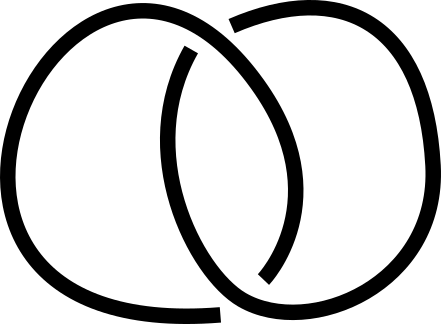},  \includegraphics[scale=0.07]{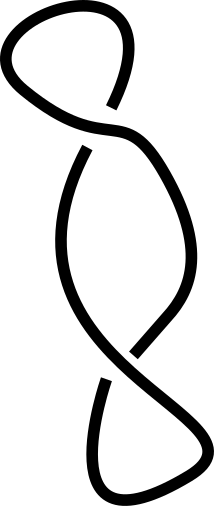},  \includegraphics[scale=0.07]{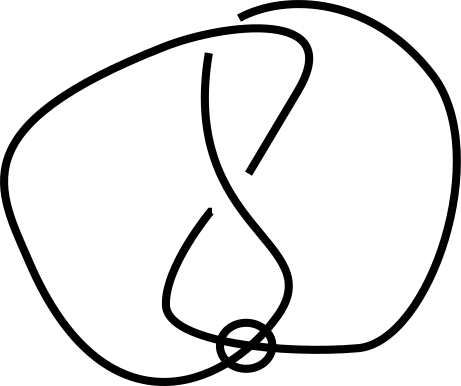}} \right\}$. Let $\mathsf{h}$ denote the height invariant of linkoids. Then, $\mathsf{h}\left(  \raisebox{-7pt}{ \includegraphics[scale=0.07]{fig/hc.png}}\right) = \mathsf{h}\left( \raisebox{-7pt}{\includegraphics[scale=0.07]{fig/hairpin.png}}\right) = 0$ and $\mathsf{h}\left(  \raisebox{-7pt}{ \includegraphics[scale=0.07]{fig/vir-tref1.png}}\right) = 1$. Hence, $\mathsf{h}\left(\operatorname{spec}\left(  \raisebox{-7pt}{\includegraphics[scale=0.05]{fig/hopf-open.png}} \right)\right) = \{ 0, 0, 1 \}$ and $\operatorname{avg}_{\mathsf{h}}\left(  \raisebox{-7pt}{\includegraphics[scale=0.05]{fig/hopf-open.png}} \right)  = 1/3$.    
\end{example}

\begin{corollary}
If two linkoid diagrams, $L_1$ and $L_2$, are equivalent up to Reidemeister moves, then $\mathcal{F}(\operatorname{spec}(L_1)) = \mathcal{F}(\operatorname{spec}(L_2))$ and $\operatorname{avg}_{\mathcal{F}}(L_1)=\operatorname{avg}_{\mathcal{F}}(L_2)$, where $\mathcal{F}$ is an invariant of virtual links.
\end{corollary}
\begin{proof}
    It follows from Corollary \ref{corr-equiv-diag-spec} and Theorem \ref{thm-new-F-inv}.
\end{proof}

\begin{corollary}
  Let $L_1$ and $L_2$ be two arbitrary linkoid diagrams and $\mathcal{F}$ be an invariant of virtual links.
  If $\mathcal{F}(\operatorname{spec}(L_1)) \neq \mathcal{F}(\operatorname{spec}(L_2))$, then  $L_1$ and $L_2$ are distinct. Similarly, if $\operatorname{avg}_{\mathcal{F}}(L_1) \neq \operatorname{avg}_{\mathcal{F}}(L_2)$, then  $L_1$ and $L_2$ are distinct.
\end{corollary}
\begin{proof}  It follows from Theorem \ref{thm-new-F-inv}.
\end{proof}

\begin{corollary}
    The $\mathcal{F}$-invariant of a linkoid over its virtual spectrum is stronger than the invariant $\mathcal{F}$ applied to any individual closure unless the invariant is a complete classifier of virtual knots/links. Similarly, the average $\mathcal{F}$-invariant, $\operatorname{avg}_{\mathcal{F}}$, is a stronger classifier of a linkoid  than the invariant $\mathcal{F}$ applied to any individual closure, unless the invariant is a complete classifier of virtual knots/links.
\end{corollary}
\begin{proof}  It follows from Theorem \ref{thm-new-F-inv}.
\end{proof}

Through the following example, it is shown that evaluating an invariant over the virtual spectrum can better tell apart linkoids compared to evaluating the invariant for any particular closure permutation. 

\begin{example}
 Let us consider the linkoids, $\displaystyle  L_1 = \raisebox{-10pt}{\includegraphics[scale=0.2]{fig/kishino-linkoid.png}}$ and $\displaystyle  L_2 = \raisebox{-10pt}{\includegraphics[scale=0.4]{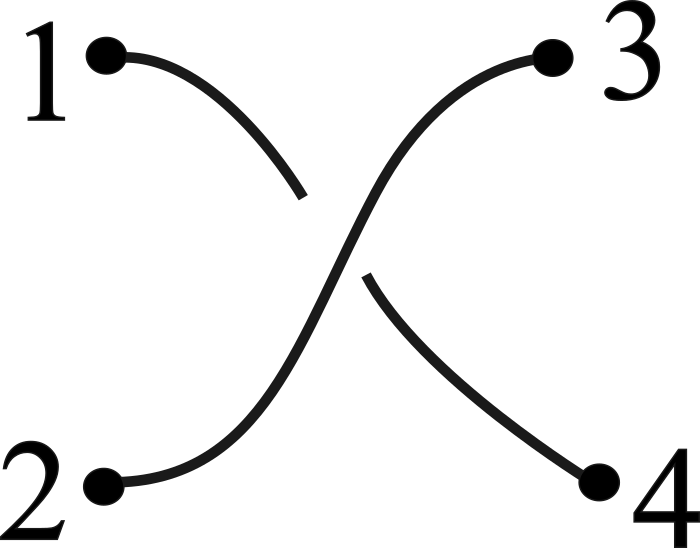}}$, each with strand permutation $\tau=(1 \quad 4)(2 \quad 3)$. Note that the Jones polynomial of both $L_1$ and $L_2$, with respect to closure permutation, $\sigma=(1 \quad 2)(3 \quad 4)$, is equal to $1$. However, the spectrum of the Jones polynomial  of $L_1$ over all closure permutations is $\{ 1,\quad 2-2A^{-2}-A^2-A^{-4},\quad -2A^4+2A^{-2}+2-2A^2-2A^4\}$. Whereas,  the spectrum of the Jones polynomial  of $L_2$ over all closure permutations is $\{ 1,\quad -A^4-A^2,\quad 1\}$. Since the spectrum of the Jones polynomial for the two linkoids are different, hence $L_1$ and $L_2$ are non-equivalent as linkoids.
\end{example}

In addition to the average spectral $\mathcal{F}$-invariant of a linkoid, the \textit{minimal spectral $\mathcal{F}$-invariant of a linkoid} can be defined as follows for any $\mathcal{F}$-invariant of virtual links, which takes values in $\mathbb{R}$. 

\begin{definition}(\textit{Minimal spectral $\mathcal{F}$-invariant of a linkoid}) Let $\mathcal{F}$ be an invariant of virtual links which takes values in $\mathbb{R}$ (such as height, genus and odd writhe). The \textit{minimal spectral $\mathcal{F}$-invariant} of a linkoid, $L$, is defined as the minimum of the values of the $\mathcal{F}$-invariant over all elements in the virtual spectrum, $\operatorname{spec}(L)$:
\begin{equation*}\displaystyle
    \operatorname{min}_{\mathcal{F}}(L) := \operatorname{min} \left\{\mathcal{F}(\ell) : \ell \in \operatorname{spec}(L)\right\} .
\end{equation*}
\label{def-min-spec-F}
\end{definition}

\begin{remark}
    The minimal spectral height and genus of a linkoid can be defined due to Definition \ref{def-min-spec-F} and they take values in $\mathbb{Z}$. Whereas, the average spectral height and genus of a linkoid takes values in $\mathbb{Q}$.
\end{remark}

\section{Measures of entanglement of Open Curves in 3-space via Virtual Closure}
\label{sec-open-curves}

In this section, it is shown how topological invariants from virtual knot theory can be extended to define measures of entanglement of open curves in 3-space. 

For a collection of open or closed curves in 3-space in general position, any (regular) projection of these curves gives rise to a linkoid diagram. In fact, any such projection will be generic with probability one. Based on the directions of projection, given as unit vectors in $S^2$, it is possible to get non-equivalent linkoid diagrams for a collection of open curves in 3-space. 
To fully capture the entanglement complexity of a collection of open curves in 3-space, it is necessary to account for all possible directions of projection \cite{Barkataki2022, Panagiotou2021,
 Panagiotou2020b, Panagiotou2015}.

In the following discussion, let $\mathcal{L}$ be a collection of $n$ open curves in 3-space, with endpoints denoted by labels from the set $E=\{1, 2, \cdots, 2n\}$, without repetition. A \textit{strand permutation} and a \textit{closure permutation} of  $\mathcal{L}$ can be defined in the same way as Definitions \ref{def-link-perm} and \ref{def_EP}, namely:

\begin{definition}\label{def-open-link-perm}(\textit{Strand permutation of collections of open curves in 3-space})  The \textit{strand permutation} of $\mathcal{L}$ is defined to be the element, $\tau \in S_{2n}$, such that, for any $i\in E$,  $i$ and $\tau(i)$ are labels for the two endpoints of the same component of $\mathcal{L}$.
\end{definition}

\begin{definition}(\textit{Closure permutation of collections of open curves in 3-space}) A \textit{closure permutation} of $\mathcal{L}_\tau$, where $\tau$ denotes a strand permutation of $\mathcal{L}$, is any element,  $\sigma \in S_{2n}$,  such that, $\sigma^2 = \operatorname{id}$ and $\sigma(i) \neq i$, for all $i \in E$.
\label{def_open-cp}
\end{definition}

Notice that, for every $\vec{\xi} \in S^2$, the projection of $\mathcal{L}$ to the plane with normal vector, $\vec{\xi}$ is a linkoid diagram, $\mathcal{L}_{\vec{\xi}}$.
A choice of strand permutation and closure permutation of a collection of open curves, $\mathcal{L}$, canonically induces a strand permutation and closure permutation on a projection, $\mathcal{L}_{\vec{\xi}}$, for all $\vec{\xi} \in S^2$. Therefore, $\displaystyle \left(\mathcal{L}_{\vec{\xi}}\right)_\tau$ and $\displaystyle \left( \left(\mathcal{L}_{\vec{\xi}}\right)_\tau , \sigma\right)_v$ are well-defined for all $\vec{\xi} \in S^2$, where $\tau$ and $\sigma$ are the given strand permutation and closure permutation of $\mathcal{L}$. The notion of virtual closure can be extended to open curves in 3-space by introducing virtual closure arcs in 3-space according to some closure permutation (See Figure \ref{open_v}), as defined before.

\begin{figure}[ht!]
    \centering
    \includegraphics[scale=0.35]{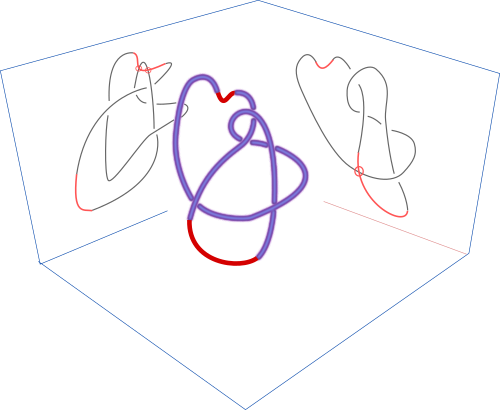}
    \caption{Projections of the virtual closure of open curves in 3-space giving rise to virtual knot diagrams.}
    \label{open_v}
\end{figure}

\begin{definition}(\textit{Virtual arc in 3-space})
A virtual arc in 3-space is an arc in 3-space such that its projection in any direction involves only virtual crossings.    
\end{definition}

\begin{definition}(\textit{Virtual closure of collections of open curves in 3-space with respect to $\sigma$}) Let $\mathcal{L}$ be a collection of open curves in 3-space with strand permutation, $\tau$. The \textit{virtual closure} of $\mathcal{L}_\tau$, with respect to $\sigma$, is denoted by $\displaystyle \left(\mathcal{L}_\tau, \sigma \right)_v$ and it is defined to be the knot/link in 3-space, with 
virtual closure arcs between $i$ and $\sigma(i)$, for all $i\in E$.  
\label{def-vclo-open}
\end{definition}

\begin{corollary}\label{corr-proj-vk-vk3} Let $\mathcal{L}$ denote a collection of open curves in 3-space and let $\tau$ and $\sigma$ denote its strand permutation and closure permutation, respectively.
Let $\vec{\xi} \in S^2$ and let $\displaystyle \left(\left(\mathcal{L}_\tau,\sigma\right)_v\right)_{\vec{\xi}}$ be the projection, with respect to $\vec{\xi}$, of the virtual closure of $\mathcal{L}_\tau$, with respect to $\sigma$ and $\displaystyle \left( \left(\mathcal{L}_{\vec{\xi}}\right)_\tau , \sigma\right)_v$ be the virtual closure of the linkoid diagram, $\displaystyle \left(\mathcal{L}_{\vec{\xi}}\right)_\tau$, with respect to $\sigma$. Then $\displaystyle \left(\left({\mathcal{L}_\tau},\sigma\right)_v\right)_{\vec{\xi}}= \displaystyle \left( \left(\mathcal{L}_{\vec{\xi}}\right)_\tau , \sigma\right)_v$. 
\begin{proof}
By Definition \ref{def-vclo-open}, all the virtual crossings in $\displaystyle \left(\left({\mathcal{L}_\tau},\sigma\right)_v\right)_{\vec{\xi}}$ are contributed by the projections of the virtual closure arcs in $\displaystyle \left({\mathcal{L}_\tau},\sigma\right)_v$. The linkoid diagram,  $\displaystyle \left(\mathcal{L}_{\vec{\xi}}\right)_\tau$, is obtained if the arcs between $i$ and $\sigma(i)$ are removed from the diagram, $\left(\left({\mathcal{L}_\tau},\sigma\right)_v\right)_{\vec{\xi}}$, for all $i \in \{1,2, \ldots, 2n\}$. Hence, $\displaystyle \left(\left({\mathcal{L}_\tau},\sigma\right)_v\right)_{\vec{\xi}}= \displaystyle \left( \left(\mathcal{L}_{\vec{\xi}}\right)_\tau , \sigma\right)_v$.
\end{proof}
\end{corollary}

The following corollary asserts that the virtual closure of collection of open curves in 3-space is independent of the embedding of the virtual arcs in 3-space.

\begin{corollary}
Let $\mathcal{L}$ denote a collection of open curves in 3-space and let $\tau$ and $\sigma$ denote its strand permutation and closure permutation, respectively. Let  $X$ and $Y$ denote two representations of $\displaystyle \left(\mathcal{L}_\tau, \sigma \right)_v$, corresponding to different embeddings of the virtual arcs. Then, $X_{\vec{\xi}} = Y_{\vec{\xi}}$ for all $\vec{\xi} \in S^2$.
\end{corollary}
\begin{proof}
    It follows from Corollary \ref{corr-proj-vk-vk3}.
\end{proof}

A projection of a collection of open curves in 3-space can give only finitely many knotoids/linkoids \cite{Panagiotou2015, Barkataki2022}. Similarly, for any closure permutation, the virtual closure of a collection of open curves in 3-space can be associated with a finite family of virtual knot diagrams (which arise from projections of the virtual closure).

\begin{corollary} Let $\mathcal{L}$ be a collection of open curves in 3-space.
 Given a closure permutation, $\sigma$, the number of distinct virtual knots associated with the virtual closure of $\mathcal{L}$ is less than or equal to the number of distinct knotoids/linkoids associated with  $\mathcal{L}$.
\label{vc_lemma}
\end{corollary}
\begin{proof}
    A projection of $\mathcal{L}$ can give only finitely many knotoids/linkoids \cite{Panagiotou2015, Barkataki2022}. Hence, a projection of the virtual closure of $\mathcal{L}$, with respect to $\sigma$, can give only finitely many virtual knots/links. The claim follows from the fact that the virtual closure map on linkoids is not injective. (See Theorem \ref{surj-thm}).
\end{proof}

\begin{remark}
 By Definition \ref{def-vclo-open}, a new notion of virtual links in 3-space arises i.e. the links in 3-space containing one or more virtual arcs. Let us denote virtual links in 3-space by $\mathcal{V}^{3D}$.  Let for every $n \in \mathbb{N}$, $\mathcal{A}_n^{3D}$ be the set of all collections of open curves in 3-space with $n$ components and $H_n = \{ \sigma \in S_{2n} : \sigma^2(i) = i, \sigma(i) \neq i , \quad \forall i \in E \}$, where $E=\{ 1, 2, \ldots , 2n\}$.  Similar to Definition \ref{def-closure-map}, we can define a closure map relating open curves in 3-space to virtual links in 3-space as follows :
 \begin{equation*}
 \begin{split}
     \phi &:  \quad \{ (\mathcal{L},\tau,\sigma)  \mid  (\mathcal{L},\tau,\sigma) \in \mathcal{A}_n^{3D} \times H_n \times H_n \text{ for some }n \in \mathbb{N}\} \longrightarrow \mathcal{V}^{3D}\\
     & \hspace{5cm} (\mathcal{L}, \tau, \sigma) \mapsto (\mathcal{L}_\tau,\sigma)_v \hspace{0.5cm},   \\
     \end{split}
 \end{equation*}
where $\tau, \sigma \in H_n$, $\mathcal{L} \in \mathcal{A}_n^{3D}$ and $(\mathcal{L}_\tau,\sigma)_v$ is the virtual closure of $\mathcal{L}_\tau$ with respect to $\sigma$. The map $\phi$ has properties analogous to properties i-iv of the virtual closure map $\varphi$ of linkoids (See Defintion \ref{def-closure-map}). 
\end{remark}

\begin{remark}
Through virtual closure, open curves in 3-space can be associated to a collection of embeddings of links in thickened surfaces of arbitrary genus.
\end{remark}

A collection of new measures of entanglement of collections of open curves in 3-space can be defined by using invariants of virtual knots as follows :

\begin{definition}(\textit{New measures of entanglement of collections of open curves in 3-space}) Let $\mathcal{L}$ be a collection of open curves in 3-space and let  $\mathcal{F}$ be any invariant of virtual links.
 Given a closure permutation $\sigma$, a measure of entanglement of $\mathcal{L}_\tau$ is,
\begin{equation}\displaystyle
\mathcal{F}_\sigma({\mathcal{L}_\tau})\hspace{0.05cm} :=\hspace{0.05cm}\frac{1}{4 \pi} \int_{\vec{\xi}\in S^2} \mathcal{F}\left(\left(\mathcal{L}_\tau, \sigma \right)_{\vec{\xi}}\right) dS\hspace{0.05cm}.
    \label{inv_open_curves}
\end{equation}
 \noindent where $\left(\mathcal{L}_\tau, \sigma \right)_{\vec{\xi}}$ denotes the projection along $\vec{\xi}$ of the virtual closure of $\mathcal{L}_\tau$, with respect to $\sigma$. 
\label{thm-inv-open-curves}
\end{definition}

Note that the integral in Equation \ref{inv_open_curves} is taken over all vectors $\vec{\xi} \in S^2$ except a set of measure zero (corresponding to the irregular projections). For collections of open curves in 3-space,  $\mathcal{F}_\sigma$ has the following properties :

\begin{enumerate}
    \item $\mathcal{F}_\sigma$ does not depend on any particular projection of the collection of open or closed curves.
    \item For a collection of open curves,  $\mathcal{F}_\sigma$ is not an invariant of a corresponding/approximating link, linkoid or virtual link.
    \item $\mathcal{F}_\sigma$ is not a topological invariant, but it is a continuous function of the curve coordinates  (see Corollary \ref{corr-tends-to-coincide}).
    \item For a collection of closed curves in 3-space (a link), $\mathcal{F}_\sigma$ coincides with a link invariant and it can be computed from a single projection, i.e.  $\mathcal{F}_\sigma\left( \mathcal{L}_\tau\right)=\mathcal{F}_\sigma\left( \left(\mathcal{L}_{\vec{\xi}}\right)_\tau\right) = \mathcal{F}\left(\left(\mathcal{L}_\tau, \sigma \right)_{\vec{\xi}}\right)$, where $\vec{\xi} \in S^2$ is any projection vector.
    \item As the endpoints of a collection of open curves in 3-space tend to coincide, according to $\sigma$, the value of $\mathcal{F}_\sigma$ tends to the value of $\mathcal{F}$ on the corresponding classical link in 3-space (see Corollary \ref{corr-tends-to-coincide}).
\end{enumerate}

\begin{corollary}\label{corr-tends-to-coincide} Let $\mathcal{L}$ denote a collection of open curves in 3-space. Then, $\mathcal{F}_\sigma$ is a continuous function of the curve coordinates of $\mathcal{L}_\tau$. As the endpoints of $\mathcal{L}_\tau$ tend to coincide according to $\sigma$, to form a closed link in 3-space, $\mathcal{F}_\sigma$ tends to the value of $\mathcal{F}$ on the corresponding classical link in 3-space.
\end{corollary}

\begin{proof}
Let us approximate $\mathcal{L}$ by a set of polygonal curves of $n$ edges each, namely $\mathcal{L}^{(n)}$. Then, 
\begin{equation}\displaystyle
\mathcal{F}_\sigma\left({\mathcal{L}^{(n)}}_\tau\right) = \sum_{i=1}^k p_i \mathcal{F}_\sigma\left({\mathsf{L}^{(n)}_i}_\tau\right)  ,
    \label{finte_sum_Fsig}
\end{equation}

\noindent where ${\mathsf{L}^{(n)}_i}_\tau$ for $i=1, \dotsc,k$ are the possible linkoids that can occur in all projections of ${\mathcal{L}^{(n)}}_\tau$ and $p_i$, the corresponding geometric probabilities. 
The geometric probability $p_i$ can be expressed as,
$\displaystyle p_i = \frac{2A_0}{4\pi}$,
where $A_0$ is the area on the sphere corresponding to unit vectors such that the projection of ${\mathcal{L}^{(n)}}_\tau$ along such vectors results in the linkoid ${\mathsf{L}^{(n)}_i}_\tau$. $A_0$ is the area of the quadrangle bounded by great circles defined by the edges and vertices of the polygonal curves in ${\mathcal{L}^{(n)}}_\tau$, which is a continuous function of the coordinates of ${\mathcal{L}^{(n)}}_\tau$ (see proof of Lemma 3.1 in \cite{Panagiotou2020b}). The result follows as $n$ goes to infinity.

As the endpoints tend to coincide, the geometric probability that a projection gives a pure linkoid (or a pure virtual link) goes to $0$, while the geometric probability that it gives the link-type linkoid goes to $1$. Thus, the value of $\mathcal{F}_\sigma$ on $\mathcal{L}_\tau$ tends to the value of $\mathcal{F}$ on the corresponding classical link in 3-space.
\end{proof}

For example, the \textit{height}, \textit{genus} and \textit{odd-writhe} for a collection of open curves in 3-space can be defined as follows :

\begin{definition}\textit{(Height of a collection of open curves in 3-space with respect to $\sigma$)} Let $\mathcal{L}$ denote a collection of open curves in 3-space. Given a closure permutation $\sigma$, the \textit{height} of a collection of open curves in 3-space, $\mathcal{L}_\tau$, is defined as,
\begin{equation}\displaystyle
h_\sigma({\mathcal{L}_\tau})\hspace{0.05cm} :=\hspace{0.05cm}\frac{1}{4 \pi} \int_{\vec{\xi}\in S^2} h\left(\left(\mathcal{L}_\tau, \sigma \right)_{\vec{\xi}}\right) dS\hspace{0.05cm}.
    \label{hgt_open_curves}
\end{equation}
 \noindent where $\left(\mathcal{L}_\tau, \sigma \right)_{\vec{\xi}}$ denotes a projection of the virtual closure of $\mathcal{L}_\tau$ with respect to $\sigma$. 
\end{definition}

\begin{definition}(\textit{Genus of  a collection of open curves in 3-space  with respect to $\sigma$} ) Let $\mathcal{L}$ denote a collection of open curves in 3-space.
  Given a closure permutation $\sigma$, the \textit{genus} of a collection of open curves in 3-space, $\mathcal{L}_\tau$, is defined as,
\begin{equation}\displaystyle
g_\sigma({\mathcal{L}_\tau})\hspace{0.05cm} :=\hspace{0.05cm}\frac{1}{4 \pi} \int_{\vec{\xi}\in S^2} g\left(\left(\mathcal{L}_\tau, \sigma \right)_{\vec{\xi}}\right) dS\hspace{0.05cm}.
    \label{genus_open_curves}
\end{equation}
 \noindent where $\left(\mathcal{L}_\tau, \sigma \right)_{\vec{\xi}}$ denotes a projection of the virtual closure of $\mathcal{L}_\tau$ with respect to $\sigma$. 
\end{definition}

\begin{definition}(\textit{Odd writhe of a collection of open curves in 3-space  with respect to $\sigma$})
Let $\mathcal{L}$ denote a collection of open curves in 3-space and let $\sigma$ be a closure permutation such that, $(\mathcal{L}, \sigma)_v$ has only one component. Then, the \textit{odd-writhe} of a collection of open curves in 3-space, is defined as,
\begin{equation}\displaystyle
\mathsf{owr}_\sigma({\mathcal{L}_\tau})\hspace{0.05cm} :=\hspace{0.05cm}\frac{1}{4 \pi} \int_{\vec{\xi}\in S^2} \mathsf{owr}\left(\left(\mathcal{L}_\tau, \sigma \right)_{\vec{\xi}}\right) dS\hspace{0.05cm}.
    \label{odd-wr_open_curves}
\end{equation}
\noindent where $\left(\mathcal{L}_\tau, \sigma \right)_{\vec{\xi}}$ denotes a projection of the virtual closure of $\mathcal{L}_\tau$ with respect to $\sigma$.
\end{definition}

Similarly, we can define the Jones polynomial and the arrow polynomial of collections of open curves in 3-space with respect to a given closure permutation. The affine-index polynomial of collections of open curves in 3-space is also well-defined provided the closure permutation in concern gives rise to a virtual link in 3-space with only 1 component.

Notice that for open curves in 3-space, the height, genus and odd-writhe are real numbers that are continuous functions of the curve coordinates and they tend to the integer topological invariant as the endpoints coincide. The Jones polynomial, arrow polynomial and affine-index polynomial for open curves in 3-space have real valued coefficients which are continuous functions of the curves co-ordinates. As the endpoints of the curves tend to coincide, these polynomials tend to the corresponding polynomial-type invariants of the resultant closed knot/link in 3-space.

\begin{remark}

 The virtual closure of open curves in 3-space can be useful in the context of physical systems that employ periodic boundary conditions \cite{Morton2009, Fukuda2023, Barkataki2023pbc, Evans2013}. The generating cell, $C$, of a PBC system is in fact an $n$-tangle, $T$, possibly involving endpoints inside the tangle. Here $n \in \mathbb{N}$ determines the number of arcs inside $C$ and the periodic conditions canonically determine a strand permutation, $\tau_{pbc}$, for $T$. The virtual closure of $T$, denoted $(T_{\tau_{pbc}}, \tau_{pbc})$, can be formed by introducing closure arcs to the endpoints in $T$, lying on the boundary surface of the $C$, with respect to $\tau_{pbc}$. In any projection of $(T_{\tau_{pbc}}, \tau_{pbc})$, the arcs inside $C$ contributes classical crossings to the diagram whereas the arcs (virtual closure arcs) lying outside $C$ contributes virtual crossings to the diagram (See Figures \ref{per_vc} and \ref{per_vc-2pbc}). Thus, for any projection of a PBC box, we can assign a virtual link. The virtual link however, depends on the direction of projection of the box.
\begin{figure}[ht!]
    \centering
    \raisebox{-20pt}{\includegraphics[scale=0.55]{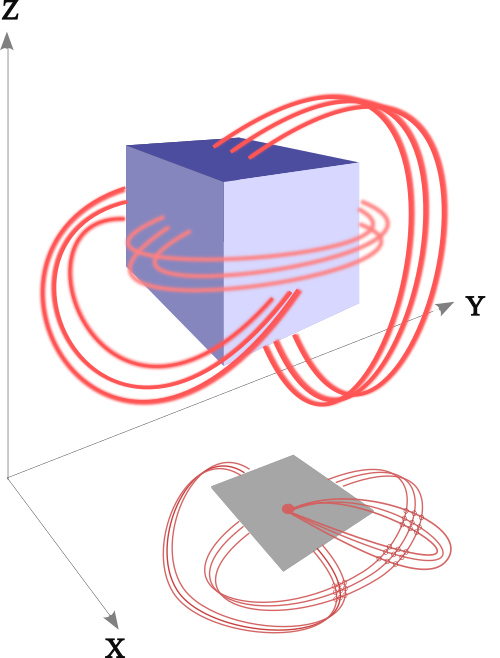}} \hspace{1cm}
    \includegraphics[scale=0.5]{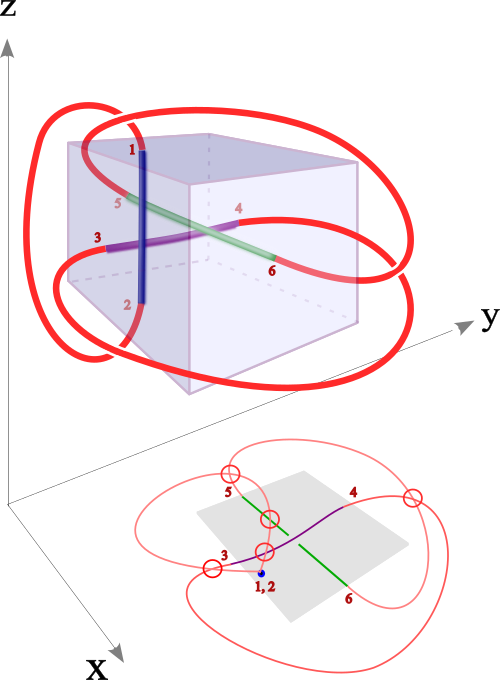} 
    \caption{Virtual closure of the arcs in a base cell of a PBC system in 3D and one of its projections (virtual knot diagram).}
    \label{per_vc}
\end{figure}
By Theorem \ref{thm-inv-open-curves}, novel measures of entanglement for systems employing PBC  can be defined in terms of the virtual closure (respecting PBC) of the arcs within a base cell.  

\begin{figure}[ht!]
    \centering
    \includegraphics[scale=0.3]{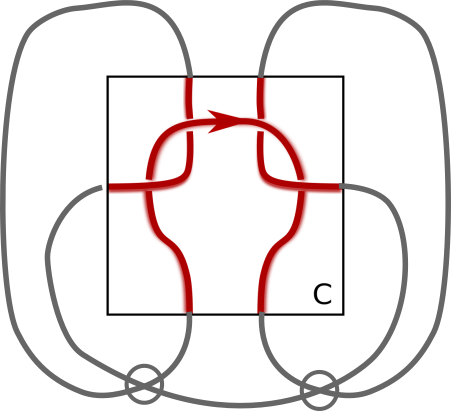} \hspace{2cm} \includegraphics[scale=0.3]{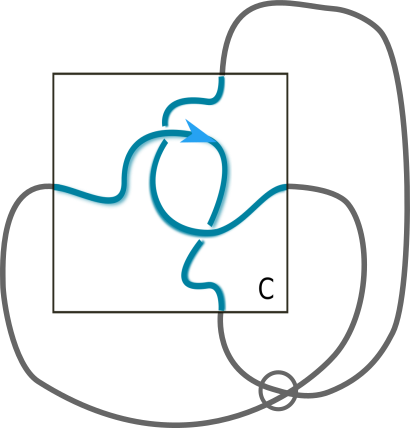}
    \caption{The virtual closure of the generating cell of two different systems employing 2 periodic boundary conditions.}
    \label{per_vc-2pbc}
\end{figure}
\end{remark}

\subsection{The Virtual Spectrum of collections of Open Curves in 3-space}

In this section, the virtual spectrum and the weighted virtual spectrum of collections of open curves in 3-space is discussed to define a new measure of entanglement of open curves in 3-space that is independent of any virtual closure.

\begin{definition}(\textit{Virtual Spectrum of open curves in 3-space})
Let  $\mathcal{L}$ be a collection of open curves in 3-space with $n$ components and strand permutation, $\tau$ on the set of endpoints, $E$. The \textit{virtual spectrum} of  $\mathcal{L}_\tau$ is defined to be the set, $\displaystyle \operatorname{spec}(\mathcal{L}_\tau) := \displaystyle \left\{(\mathcal{L}_\tau, \sigma)_v : \sigma \in H_n \right\}$, where  $H_n$ is the set of all closure permutations on $E$.
\end{definition}

\begin{definition}(\textit{Weighted Virtual Spectrum of open curves in 3-space})
Let  $\mathcal{L}$ be a collection of open curves in 3-space with $n$ components and strand permutation, $\tau$ on the set of endpoints, $E$. The \textit{weighted virtual spectrum} of  $\mathcal{L}_\tau$ is defined to be the set,

$\displaystyle \operatorname{spec}(\mathcal{L}_\tau) := \displaystyle \left\{\left(w_\sigma,(\mathcal{L}_\tau, \sigma)_v\right) : \sigma \in H_n ,\quad w_\sigma = \frac{1}{2n}\sum_{i\in E}\operatorname{d}(i,\sigma(i))\right\}$, where  $H_n$ is the set of all closure permutations on $E$ and $\operatorname{d}(i,\sigma(i))$ denotes the  Euclidean distance between the points $i$ and $\sigma(i)$ for all $i \in E$.
\end{definition}

Using the weighted virtual spectrum, we define a new measure of  entanglement for collections of open curves in 3-space:

\begin{definition}(\textit{Spectral measure of open curves in 3-space})
Let  $\mathcal{L}$ be a collection of open curves in 3-space with $n$ components and strand permutation, $\tau$ on the set of endpoints, $E$. Let $\mathcal{F}$ be an invariant of virtual knots/links. The $\mathcal{F}$-spectral measure of $\mathcal{L}_\tau$ is defined as:

$$\displaystyle \mathcal{F}(\mathcal{L}_\tau) :=\sum_{\sigma \in H_n} \frac{w_{min}}{w_\sigma} \mathcal{F}_\sigma\left(({\mathcal{L}}_\tau, \sigma)_v\right) , $$
 where  $\displaystyle w_\sigma = \frac{1}{2n}\sum_{i\in E}\operatorname{d}(i,\sigma(i))$ and $w_{min} = \operatorname{min}\{ w_\sigma \mid \sigma \in H_n \}$.
\end{definition}

Let $\mathcal{F}$ be an invariant of virtual knots/links. For collections of open curves in 3-space, the $\mathcal{F}$-spectral measure has the following properties :

\begin{enumerate}
    \item It is independent of any choice of projection of the collection of open curves in 3-space.
    \item It is not an invariant of a corresponding/approximating classical or virtual link.
    \item It is not a topological invariant, but it is a continuous function of the curve coordinates.
    \item As the endpoints of a collection of open curves in 3-space tend to coincide, according to a particular $\sigma$, its $\mathcal{F}$-spectral measure tends to the $\mathcal{F}$ invariant of the corresponding classical link in 3-space (see Theorem \ref{thm-vko-tends}).
\end{enumerate}

In particular, we can define the spectral height, the spectral genus, the spectral Jones polynomial and the spectral arrow polynomial for collections of open curves in 3-space. The spectral height and spectral genus take values in $\mathbb{R}$ whereas  the spectral Jones polynomial and the spectral arrow polynomial are polynomial with coefficients in $\mathbb{R}$.

\begin{theorem}\label{thm-vko-tends}
Let  $\mathcal{L}$ be a collection of open curves in 3-space with $n$ components and strand permutation, $\tau$ on the set of endpoints, $E$. Let $\mathcal{F}$ be an invariant of virtual links.  As the endpoints tend to coincide according to a closure permutation, $\sigma$, the $\mathcal{F}$-spectral measure tends to the $\mathcal{F}$-invariant of the resultant closed link in 3-space.
\end{theorem}
\begin{proof}
    As endpoints tend to coincide, $w_{min} \rightarrow 0$ and $\mathcal{F}_\sigma\left(({\mathcal{L}}_\tau, \sigma)_v\right)$ is the only surviving term in $\mathcal{F}(\mathcal{L}_\tau)$. The claim follows from Corollary \ref{corr-tends-to-coincide}.
\end{proof}

\section{Conclusions}
\label{sec-conc}
  In this paper, new invariants of linkoids are introduced via a mapping from linkoids to virtual links. This leads to a new collection of strong invariants of linkoids that are independent of any given virtual closure and to novel measures of entanglement of open curves in 3-space. These measures are continuous functions of the curve coordinates and tend to the corresponding classical invariants when the endpoints of the curves tend to coincide. 
  
  The mapping from linkoids to virtual links gives us access to many ideas and invariants of virtual links that we can now apply to studying linkoids and open curves in 3-space. Thus the program we follow here gives new applications of virtual knot theory to classical and physical knotting.

  The association of virtual knots to linkoids and open curves in three space can be interpreted as a way to construct embedded curves in thickened surfaces that reflect these topological structures. We have concentrated on making new invariants of linkoids and open curves, but the geometry of the association with curves in thickened surfaces remains to be more fully investigated in the future.

\section{Acknowledgements}
Kasturi Barkataki and Eleni Panagiotou were supported by NSF DMS-1913180 and NSF CAREER 2047587.

\bibliographystyle{vancouver}
\bibliography{paperDatabase}

\end{document}